\documentclass{article}
\usepackage[utf8]{inputenc}
\usepackage{amsmath,amssymb,amsthm}
\usepackage{ wasysym }
\usepackage{mathrsfs} 
\usepackage{euscript}
\usepackage{color}
\usepackage[shortlabels]{enumitem}
\usepackage{bm}
\usepackage{charter}
\usepackage{comment}
\usepackage{srcltx}
\usepackage{xcolor}
\usepackage{graphicx}
\usepackage{amsmath}
\usepackage{amssymb}
\usepackage{euscript}
\usepackage{epic,eepic}
\usepackage{pstricks}
\usepackage{parskip}

\usepackage{tikz}
\usepackage{color}
\usepackage{multirow}
\usepackage{float}
\usepackage{multirow, array}
\usepackage{booktabs}
\usepackage{mathrsfs}
\usepackage{mathdots}
\usepackage{stackrel}
\usepackage{relsize}
\usepackage{ dsfont }
\usepackage{listings}
\usepackage{color}
\usepackage{graphicx}
\usepackage{setspace} 
\usepackage[utf8]{inputenc}
\usepackage{xcolor}
\usepackage[USenglish]{babel}
\usepackage{mathtools}
\usepackage{subfigure}
\usepackage{amsmath}
\usepackage{amssymb}
\usepackage{euscript}
\usepackage{epic,eepic}
\usepackage{pstricks}
\usepackage{parskip}
\usepackage{tikz}
\usepackage{multirow}
\usepackage{float}
\usepackage{array}
\usepackage{booktabs}
\usepackage{mathrsfs}
\usepackage{tabularx}
\usepackage{mathdots}
\usepackage{stackrel}
\usepackage{algorithm, algpseudocode}
\usepackage{fancybox, calc}
\usepackage{verbatim}
\usepackage{fancyhdr}
\usepackage{tabularx}

\definecolor{mygreen}{rgb}{0,0.6,0}
\definecolor{mygray}{rgb}{0.5,0.5,0.5}
\definecolor{mymauve}{rgb}{0.58,0,0.82}
\definecolor{labelkey}{rgb}{0,0.08,0.45}
\definecolor{refkey}{rgb}{0,0.6,0.0}
\definecolor{Brown}{rgb}{0.45,0.0,0.05}
\definecolor{dgreen}{rgb}{0.00,0.49,0.00}
\definecolor{dblue}{rgb}{0,0.08,0.75}
\DeclareSymbolFont{largesymbolsA}{U}{jkpexa}{m}{n}
\SetSymbolFont{largesymbolsA}{bold}{U}{jkpexa}{bx}{n}
\DeclareMathSymbol{\varprod}{\mathop}{largesymbolsA}{16}

\definecolor{labelkey}{rgb}{0,0.08,0.45}
\definecolor{refkey}{rgb}{0,0.6,0.0}
\definecolor{Brown}{rgb}{0.45,0.0,0.05}
\definecolor{dgreen}{rgb}{0.00,0.49,0.00}
\definecolor{dblue}{rgb}{0,0.08,0.75}
\RequirePackage[colorlinks,hyperindex]{hyperref} 
\hypersetup{linktocpage=true,citecolor=dblue,linkcolor=dgreen}

\newtheorem{theorem}{Theorem}[section]
\newtheorem{corollary}[theorem]{Corollary}
\newtheorem{lemma}[theorem]{Lemma}

\newtheorem{proposition}[theorem]{Proposition}

\newtheorem{assumption}[theorem]{Assumption}
\theoremstyle{definition}
\newtheorem{definition}[theorem]{Definition}
\newtheorem{remark}[theorem]{Remark}
\newtheorem{example}[theorem]{Example}

\def\acknowledgement{\par\addvspace{17pt}\small\rmfamily
	\trivlist\if!\ackname!\item[]\else
	\item[\hskip\labelsep
	{\bfseries\ackname}]\fi}

\numberwithin{equation}{section}

\newcommand{\XX}{\ensuremath \mathcal{X}}

\DeclareMathOperator*{\dom}{\ensuremath{\text{\rm dom}}}
\DeclareMathOperator*{\argmin}{\text{\rm argmin}}

\newcommand{\scal}[2]{{\left\langle{{#1},{#2}}\right\rangle}}

\newcommand{\Id}{\ensuremath{\operatorname{Id}}\,}
\newcommand{\RR}{\ensuremath{\mathbb{R}}}

\newcommand{\NN}{\ensuremath{\mathbb N}}

\newcommand{\ft}{g}
\newcommand{\fh}{q}
\newcommand{\vertiii}[1]{{\left\vert\kern-0.25ex\left\vert\kern-0.25ex\left\vert #1 
		\right\vert\kern-0.25ex\right\vert\kern-0.25ex\right\vert}}
\newcommand{\normi}{\vert\kern-0.25ex\vert\kern-0.25ex\vert}

\ProvideTextCommand{\DJ}{OT1}{\raisebox{0.25ex}{-}\kern-0.4em D}

\numberwithin{equation}{section}
\usepackage{graphicx}
\usepackage{textcomp}
\def\FF{{\mathcal{F}}}
\newcommand{\cris}[1]{\textcolor{black}{#1}}

\def\t{\overline{t}}
\usepackage{csquotes}
\usepackage{chngcntr}
\usepackage[backend=biber, style=ieee, citestyle=numeric-comp, sorting=nty]{biblatex}
\title{Delayed Feedback in Online Non-Convex Optimization: A Non-Stationary Approach with Applications}
\author{Felipe Lara\thanks{Instituto de Alta investigaci\'on (IAI), Universidad de Tarapac\'a,
Arica, Chile. E-mail: felipelaraobreque@gmail.com; flarao@academicos.uta.cl. Web:
felipelara.cl, ORCID-ID: 0000-0002-9965-0921} \and
Cristian Vega\thanks{Instituto de Alta investigaci\'on (IAI), Universidad de Tarapac\'a,
Arica, Chile. E-mail: cristianvegacereno6@gmail.com,  ORCID-ID: 0000-0001-7792-0137}`}

\date{\ttfamily \today}
\begin{document}

\maketitle
\begin{abstract}
\noindent We study non-convex online optimization problems with delay and noise by evaluating dynamic regret in the non-stationary setting when the loss functions are quasar-convex. In particular, we consider scenarios involving quasar-convex functions either with Lipschitz gradients or with weak smoothness, and in each case we establish bounded dynamic regret in terms of cumulative path variation, achieving sub-linear rates. Furthermore, we illustrate the flexibility of our framework by applying it to both theoretical settings such as zeroth-order (bandit) and practical applications with quadratic fractional functions. Moreover, we provide new examples of non-convex functions that are quasar-convex by proving that the class of differentiable strongly quasiconvex functions is strongly quasar-convex on convex compact sets. Finally, several numerical experiments validate our theoretical findings, illustrating the effectiveness of our approach.
\end{abstract}
\medskip


\section{Introduction}

Online optimization \cite{hazan2007logarithmic,zinkevich2003online,shalev2012online,hazan2016introduction} is an outstanding mathematical model for sequential decision-making, which involves continuous optimization, machine learning, and data sciences, among others. It is used to provide automated decisions based on data, and its goal is to minimize the cumulative loss (or regret) of those decisions. This model has attracted the attention of the scientific community from several research fields in view of its applications in a variety of real-world problems, such as portfolio selection \cite{agarwal2006algorithms,blum1997universal}, dynamic pricing \cite{chen2024online}, decision trees \cite{helmbold1995predicting}, machine learning \cite{blum2005line}, shortest paths \cite{kalai2005efficient}, network resource allocation \cite{wang2021delay,wang2022delay}, and generalized linear models \cite{pun2024online}, among others. 
\medskip

\noindent Each decision involves a loss determined by a function, which is revealed only after it is made. It can be viewed as a structured, repeated game between a learner and an adversary. At each round $t$, a player selects a decision $x_{t}$ from a convex feasible set $\XX \subset \RR^{p}$. Then, an adversary gives a loss or cost function $f_{t} \colon \XX \to \RR$, and the player incurs a loss $f_{t}(x_{t})$. The aim is to minimize the (stationary) regret $R^{S}_{T}$, defined as:$$R^{S}_{T} := \sum\limits_{t=1}^T f_{t}(x_{t})-\min\limits_{x \in \XX} \sum\limits_{t=1}^T f_{t}(x),$$which is the difference between the cumulative loss of the algorithm and the cumulative loss of the best fixed decision in hindsight \cris{(assuming it exists)}. Since the loss function is only revealed after the decision is made, we cannot expect to beat the opponent.  Significant regret is required to be sub-linear in $T$, ensuring that as $T$ becomes sufficiently large, the player performs as well as the best fixed decision.\medskip

\noindent \cris{A classic assumption in the literature is that loss functions are convex.} This assumption is not only due to the many applications of convex functions but also because of theoretical limitations, since when dealing with non-convex loss functions, several useful properties are no longer valid — for instance, local information is no longer global. In recent years, several efforts to deal with non-convex loss functions have been developed, mainly motivated by real-world applications in which classes of non-convex functions are more natural than convex ones. An interesting case occurs when dealing with {\it (strongly) quasar-convex} functions introduced in \cite{CEG,guminov2017accelerated,hinder2020near,hardt2016gradient} (under different names).  This class of functions provides us with several useful properties that are sufficient to ensure convergence rates of the gradient method, even when quasar-convex functions may have several local minimizers that are not necessarily global (see Example \ref{ex1}).
\medskip

\noindent\cris{ Another classical assumption is that the strategy is unchanged throughout the periods.} However, in several real-world scenarios, stationary regret is not adequate for evaluating performance. For instance, when aiming to maximize profit after several trades, the profit depends on the best investment at each round, rather than on the best investment for the sum of all loss functions at the final round. In such cases, the so-called {\it non-stationary regret} \cite{besbes2015non}, which compares the cumulative loss with the minimum of each loss, is more suitable:
\begin{align}
R^{NS}_{T}=\sum\limits_{t=1}^{T} f_{t}(x_{t})-\sum\limits_{t=1}^{T} f_{t}(x_{t}^{*}),
\end{align}
where $x_{t}^{*}$ represents the solution that minimizes $f_{t}(x)$ over $\XX$. \cris{Let  $x^{*} \in \argmin\limits_{x \in \XX} \sum\limits_{t=1}^{T} f_{t}(x) $denote a minimizer of the cumulative loss over $\XX$ (assuming it exists). Since $f_{t}(x_{t}^{*}) \leq f_{t}(x^{*})$ for every $t \in \{1,\ldots, T\}$,it follows that the non-stationary regret is always greater than or equal to the stationary regret, i.e., $R^{S}_{T} \leq R^{NS}_{T}.$} Moreover, dynamic regret has a theoretical advantage since the sum of non-convex functions may not have a unique minimizer even when each individual function has its own unique minimizer; for instance, strongly quasiconvex and strongly quasar-convex functions have unique global minimizers, but the sum of even two of these functions may have more than one minimizer. This makes dynamic regret particularly useful for addressing non-convexity.\medskip

\noindent Observe that the previous approach does not consider the potential delay between making a decision and receiving feedback, assuming that the loss function $f_t$ is revealed immediately after the player makes their decision $x_t$. Delays of this kind have been extensively studied in \cite{wan2022online,quanrud2015online,joulani2013online,wan2023non} and are common in real-world applications such as in economic analysis or online advertising \cite{he2014practical}. For example, when assessing the impact of new economic or financial data or political decisions with economic consequences, feedback is often delayed. Similarly, in online advertising, the output of the loss function may depend on whether a user clicks an ad. Users may take additional time before providing feedback, and non-feedback can only be inferred if the user does not click the ad after a sufficiently long period, which makes the situation uncertain for some time after the ad has been viewed.\medskip

\noindent Another classical extension of online optimization is the bandit setting, where the gradient of $f_t$ is approximated by a zeroth-order method. This approach is suitable for problems where the gradient is not directly available or is computationally expensive to compute, even when the function itself is explicitly known. The bandit setting has proved useful in several applications across fields such as economics, robotics, statistics, and machine learning, and its online version has been extensively studied in \cite{flaxman2004online,agarwal2010optimal,li2019bandit,saha2011improved}.\medskip

\noindent In this paper, we study online optimization in a very general setting. First, we address the non-convex case by assuming that the loss functions are not necessarily convex; in particular, we assume them to be quasar-convex (see \cite{CEG,guminov2017accelerated,hinder2020near,hardt2016gradient}), a class of functions that includes convex functions as well as several non-convex ones. Second, we incorporate a delay between the decision and the corresponding feedback, which is a realistic assumption for applied models where data requires additional time for analysis or computation. Third, we consider the presence of noise in the collected data to obtain greater flexibility and applicability of our online optimization model. Under this general framework, we establish bounded dynamic regret with sub-linear rates. Furthermore, we provide new examples of non-convex functions that fall into the class of quasar-convex functions, namely the class of strongly quasiconvex functions (Polyak \cite{P-1966}), which has been shown to be useful in several applications in mathematical models and algorithms (see \cite{grad2025strongly,iusem2024two,lara2022strongly,lara2024characterizations}). Finally, we apply our results to two important classes of (non-convex) quasar-convex functions: generalized linear models (GLM henceforth), and quadratic fractional functions, which are (strongly) quasiconvex. Several non-convex numerical experiments illustrate our findings.\medskip

\noindent The structure of the paper is as follows:In Section \ref{sec: sa}, we analyze the state of the art of the proposed problem under convexity/non-convexity assumptions and with and without delay. In Section \ref{sec:back}, we introduce the notation along with some preliminary concepts. Section \ref{sec:mr} presents the main algorithm and provides a regret bound. Additionally, we demonstrate the flexibility of our approach by extending the algorithm and its theoretical guarantees to the bandit setting. In Section \ref{sec:ea}, we present new examples of quasar-convex functions that will be used in the subsequent section. Section \ref{sec:ne} evaluates the algorithm's performance on three numerical applications: highly non-convex functions, GLM, and quadratic fractional functions. Finally, Section \ref{sec:cfw} offers conclusions and future research directions.

\section{State of the Art and Related Work}\label{sec: sa}

The standard algorithm for online optimization is the Online Gradient Descent (OGD) algorithm. At each step, the algorithm updates its decision based on the gradient of the loss function at the current decision: \begin{align*}x_{t+1} = P_{\XX} (x_{t} - \eta_{t} \nabla f_{t} (x_{t})),\end{align*}where $P_{\XX}$ denotes the projection onto the convex set $\XX$, and $\eta_{t}$ is the step-size (learning rate). Several papers have established sub-linear regret bounds for different classes of functions. For convex functions without delay, \cite{zinkevich2003online} proved a stationary regret bound of $\mathcal{O}(T^{\frac{1}{2}})$, while \cite{hazan2007logarithmic} improved this result to a regret rate of $\mathcal{O}(\ln(T))$ for strongly convex functions. In \cite{flaxman2004online} and \cite{agarwal2010optimal}, it is extended for the bandit setting for convex and strongly convex, respectively. 

\medskip    

\noindent \cris{In \cite[Theorem~2.1]{quanrud2015online}, the delayed version of the OGD algorithm (DOGD) is proposed for \cris{$L$-Lipschitz functions}. It is shown that DOGD with a constant learning rate $\eta_{t}=L^{-1}\left(T+D\right)^{-\frac{1}{2}}$ achieves a regret bound of order $\mathcal{O}(D^{\frac{1}{2}})$, where $D=\sum\limits_{t=1}^{T} d_{t}$ denotes the sum of the delay of each round $d_{t}$.} The extension of DOGD for the bandit setting was studied in \cite{li2019bandit}. Observe that the constant learning rate used in DOGD does not take advantage of the strong convexity of loss functions. In standard OGD, \cite{hazan2007logarithmic} established a $\ln(T)$ regret bound for $\beta$-strongly convex functions by setting $\eta_{t}=\frac{1}{\beta t}$. This learning rate leverages the increasing nature of the inverse of $\frac{1}{\eta_{t}}$ by the modulus of strong convexity of $f_{t}$, where $\frac{1}{\eta_{t}}-\frac{1}{\eta_{t-1}}=\beta$. Inspired by this, \cite{wan2022online} proposes an initialization where $\frac{1}{\eta_{0}}=0$. Otherwise, $\frac{1}{\eta_{t}}-\frac{1}{\eta_{t-1}}=a\beta,$ where $a$ is the number of gradients that arrive at round $t$. DOGD for strongly convex functions achieves a regret bound of $\mathcal{O}(d\ln(T))$, \cris{where $d=\max\limits_{t\in [T]} d_{t}$ denotes the maximum delay of all rounds $\left\{d_{t}\right\}_{t\in [T]}$.} The results for the stationary setting are summarized in Table \ref{tab:regret_bounds S}.\\

\noindent In \cite{besbes2015non}, it is shown that without any additional assumptions regarding how the loss functions evolve, the non-stationary regret can grow linearly with time $T$. To achieve a sub-linear non-stationary regret, it is essential to impose additional assumptions on the variation of the loss functions over time, often quantified by \cris{the cumulative variation of the loss functions $C_{T}$, defined as:\begin{equation*}  C_{T} \;=\; \sum\limits_{t=1}^{T-1} \sup_{x \in \XX} \left| f_{t}(x) - f_{t+1}(x) \right|.\end{equation*}}

\noindent Specifically, \cite{besbes2015non} establishes non-stationary regret bounds of $\mathcal{O} (T^{\frac{2}{3}}C_{T}^{\frac{1}{3}})$ for Online Gradient Descent with noisy gradient feedback for convex functions and $\mathcal{O} (T^{\frac{1}{2}}C_{T}^{\frac{1}{2}})$ for strongly convex functions. Subsequently, \cite{mokhtari2016online} improved the bound in the strongly convex case to $\mathcal{O}\!( V_{T})$, where $V_{T}$ is the cumulative path defined as follows: \begin{equation*} V_{T}=\sum\limits_{t=1}^{T-1} \|x_{t}^{*}-x_{t+1}^{*}\|.\end{equation*}For the delayed setting, \cite{wan2023non} obtained a non-stationary regret bound of $\mathcal{O}\!(\left(dTV_{T}\right)^{\frac{1}{2}} )$ for convex functions.\\ 

\medskip\noindent \cris{In \cite{kim2022online}, OGD was studied for functions satisfying the Polyak-{\L}ojasiewicz condition, while in \cite{zhang2017improved}, it was studied for functions that have quadratic growth. In \cite{zhang2015online}, compositions of increasing functions with linear terms are studied.  In \cite{gao2018online}, weakly pseudo-convex functions are considered, achieving a stationary regret bound of $\mathcal{O}(V_{T})$ without delay.}  Lastly, for quasar-convex functions, \cite{pun2024online} provided several significant results: a stationary regret bound of $\mathcal{O}(T^{\frac{1}{2}}V_{T}^{\frac{1}{2}})$ \cris{without delay}, a bound of $\mathcal{O}(V_{T})$ for weakly smooth and quasar-convex functions without delay, and another bound of $\mathcal{O}(V_{T})$ for strongly quasar-convex functions without delay. The results for the non-stationary setting are summarized in Table \ref{tab:regret_bounds}.
\begin{table}[htbp]
\centering
\begin{tabular}{|l|l|l|l|}
\hline
\textbf{Work} & \textbf{Function} & \textbf{Delay}  & \textbf{Regret Bound} \\ \hline
\cite{zinkevich2003online} & Convex & No & $\mathcal{O}(T^{\frac{1}{2}})$ \\ \hline
\cite{hazan2007logarithmic} & Strongly Convex & No & $\mathcal{O}(\ln(T))$ \\ \hline
\cite{quanrud2015online} & Convex & Yes & $\mathcal{O}(D^{\frac{1}{2}})$ \\ \hline
\cite{wan2022online} & Strongly Convex & Yes & $\mathcal{O}(d\ln(T))$ \\ \hline
\end{tabular}
\caption{Regret bounds for stationary regret}
\label{tab:regret_bounds S}
\end{table}

\begin{table}[htbp]
\centering
\begin{tabular}{|l|l|l|l|}
\hline
\textbf{Work} & \textbf{Function} & \textbf{Delay}  & \textbf{Regret Bound} \\ \hline
\cite{besbes2015non} & Convex & No  & $\mathcal{O}(T^{\frac{1}{3}}V_{T}^{\frac{2}{3}})$ \\ \hline
\cite{mokhtari2016online} & Strongly Convex & No & $\mathcal{O}(V_{T})$ \\ \hline
\cite{wan2023non} & Convex & Yes & $\mathcal{O}(\left(dTV_{T}\right)^{\frac{1}{2}})$ \\ \hline
\cite{pun2024online} & Quasar-Convex & No & $\mathcal{O}(T^{\frac{1}{2}}V_{T}^{\frac{1}{2}})$ \\ \hline
 \cite{pun2024online} & Weakly Smooth \& Quasar-Convex & No  & $\mathcal{O}(V_{T})$ \\ \hline
\cite{pun2024online} & Strongly Quasar-Convex & No & $\mathcal{O}(V_{T})$ \\ \hline
Our work & Quasar-Convex & Yes &  $\mathcal{O}(d T^{\frac{1}{2}}V_{T}^{\frac{1}{2}})$ \\ \hline
Our work & Weakly Smooth \& Quasar-Convex  & Yes &  $\mathcal{O}(dV_{T})$ \\ \hline
\end{tabular} 
\caption{Contribution and Comparison of non-stationary regret with related works}
\label{tab:regret_bounds}
\end{table}

\section{Preliminaries}\label{sec:back}
Let $f\colon \RR^{p} \to \RR$ be a real-valued function. We denote by $S_{\lambda}(f) := \{x \in \RR^{p} \mid f(x) \leq \lambda\}$ the sublevel set of $f$ at height $\lambda \in \RR$ and by $\argmin\limits_{\XX} f$ the set of all \cris{global minimizers of $f$ on $\XX$. 
Let $U\in\RR^{p}$ be a nonempty open set. Let a function $f\colon \RR^{p}\mapsto\RR$ such that at each point of the set $U$ all partial derivatives continuous. Then we say that $f$ is of the class $\mathcal{C}_{1}$ on $U$. The set of all these functions is denoted by $\mathcal{C}_{1}(U)$.}\\  \\
\noindent A function $f$ is said to be:
\begin{itemize}
 \item[$(a)$] (strongly) convex with modulus $\gamma \geq 0$ if, given any $(x, y) \in \RR^{p}\times\RR^{p}$, we have
 \begin{equation}\label{strong:convex}
  f(\lambda y+(1-\lambda)x) \leq \lambda f(y)+(1-\lambda) f(x) -
  \lambda (1-\lambda) \frac{\gamma}{2}\|x-y\|^{2},
 \end{equation}

 \item[$(c)$] (strongly) quasiconvex with modulus $\gamma \geq 0$ if, given any 
 $(x, y) \in \RR^{p}\times\RR^{p}$, we have
\begin{equation}\label{strong:quasiconvex}
  f(\lambda y+(1-\lambda)x) \leq \max \{f(y), f(x)\}-\lambda(1 -
  \lambda) \frac{\gamma}{2}\|x-y\|^{2}.
 \end{equation}
 \end{itemize}
When $\gamma = 0$, \cris{we simply say} that the function $f$ is convex (resp. quasiconvex).
Furthermore, note that every (strongly) convex function is (strongly) quasiconvex, 
while the reverse statements does not hold in general (see \cite{CM-Book,HKS,lara2022strongly}).
Let us recall the following results.
\medskip
\begin{lemma}\label{exist:unique} {\rm (\cite[Corollary 3]{lara2022strongly})}
 Let $\XX \subseteq \RR^{p}$ be a closed and convex set and $f\colon \XX \rightarrow \RR$ be a lsc and strongly qua\-si\-con\-vex function on $\XX$ with modulus $\gamma> 0$. Then $\argmin\limits_{\XX}\,f$ is a singleton.
\end{lemma}
\noindent Furthermore, the unique minimizer $x^{*} \in \XX$ of the strongly quasiconvex function $f$ satisfies a quadratic growth condition (see 
\cite{nam2024strong}):
\begin{equation}\label{qgc}
 f(x^{*}) + \frac{\gamma}{4} \|x^{*}-y\|^{2} \leq f(y), ~ \forall ~ y \in \XX.
\end{equation}

\noindent For differentiable strongly quasiconvex functions, we have the following characterization \cite{VNC-2}.
\medskip
\begin{lemma}\label{char:gradient} {\rm (\cite[Theorems 1 and 6]{VNC-2})}
 Let $\XX \subseteq \RR^{p}$ be a convex and open set and $f\colon \XX \rightarrow 
 \RR$ be a differentiable function. Then $f$ is strongly quasiconvex if and only if 
 there exists $\gamma \geq 0$ such that for every $(x, y) \in \XX^{2}$, we have
\begin{equation*}
  f(x) \leq f(y) ~ \Longrightarrow ~ \scal{\nabla f(y)}{x-y}
  \leq -\frac{\gamma}{2}\|x-y\|^{2}.
 \end{equation*}
\end{lemma}

\noindent The previous result extends the well-known characterization for 
differentiable quasiconvex functions given by \cite{arrow1961quasi} 
($\gamma=0)$.\\ 
\noindent Before continuing, we present some examples of strongly quasiconvex functions that are not (necessarily) convex. 

\medskip

\begin{remark}\label{rem:exam}
 \begin{enumerate}
  \item \cris{Let $f\colon \RR^{p} \rightarrow \RR$ be given 
   by $f(x) = \|x\|^{\alpha}$, with $\alpha \in \,]0, 1[$. Clearly, $f$ is non-convex, but it is strongly quasiconvex on any $\mathbb{B} (0, r)$, $r > 0$, by \cite[Corollary 3.9]{nam2024strong} ($0<\alpha<1$).}

  \item\label{3.3 (ii)} Let $A, B \in \RR^{p\times p}$, $a, b \in 
  \RR^{p}$, $\alpha, \beta \in \RR$, and $f\colon 
  \RR^{p} \rightarrow \RR$ be the functions given by:
  \begin{align}\label{ex:f3}
   f(x)=\frac{\ft(x)}{\fh(x)}=\frac{\frac{1}{2} \scal{Ax}{x}+\scal{a}{x}+\alpha}{\frac{1}{2} \scal{Bx}{x}+\scal{b}{x}+\beta}.
  \end{align}
  Take $0 < m < M$ and define:
  $$K := \{x \in \RR^{p}\mid ~ m \leq \fh(x) \leq M\}.$$ 
   If $A$ is a positive definite matrix and at least one of the following 
   conditions holds:
  \begin{enumerate}
   \item[$(a)$] $B=0$ (the null matrix),
 
   \item[$(b)$] $\ft$ is nonnegative on $K$ and $B$ is negative semi-definite,
 
   \item[$(c)$] $\ft$ is non-positive on $K$ and $B$ is positive semi-definite,
 \end{enumerate}
  then $f$ is strongly quasiconvex on $K$ with modulus $\gamma =
  \frac{\sigma_{\min} (A)}{M}$ by \cite[Proposition 4.1]{iusem2024two}, where 
  $\sigma_{\min} (A)$ is the minimum eigenvalue of $A$.

  \item Let $\ft, \fh\colon \RR^{p} \rightarrow \RR$ 
  be two strongly quasiconvex functions with modulus $\gamma_{1}, 
  \gamma_{2} > 0$, respectively. Then $f := \max\{\ft, \fh\}$ is 
  strongly quasiconvex with modulus $\gamma := \min\{\gamma_{1}, 
  \gamma_{2}\} > 0$ (see \cite{VNC-2}).


  \item If $A: \RR^{p} \rightarrow \RR^{p}$ is a linear operator and $f$ is a strongly quasiconvex function with modulus $\gamma \geq 0$, then $g := f \circ A$ is strongly quasiconvex with modulus $\gamma \sigma_{\min}(A) \geq 0$ \cite{grad2025strongly} (where $\sigma_{\min}(A)$ is the minimum positive eigenvalue of $A$).
 \end{enumerate}
\end{remark}

\medskip

\begin{definition}(See, for instance,\cite[Definition 1]{hinder2020near})
Let $\kappa \in ]0, 1]$ and $f \colon \RR^{p} \to \RR$ be a continuously differentiable function. 
Then $f$ is said to be $(\kappa, \gamma)$-strongly quasar-convex with respect to $x^{*} \in \argmin\limits_{\XX}\,f$ on $\XX$ if
\begin{align*}
 f(x^{*}) \geq f(x) + \frac{1}{\kappa} \scal{\nabla f(x)}{x^{*} - x} + \frac{
 \gamma}{2} \|x^{*} - x\|^{2}, ~ \forall ~ x \in \XX. 
\end{align*}
When $\gamma = 0$, \cris{we simply say} that $f$ is $\kappa$-quasar-convex on $\XX$. 
\end{definition}

\medskip
\noindent  If $\kappa=1$, then a function is $1$-quasar-convex if and only if it is star-convex (see \cite{nesterov2006cubic}). In particular, every strongly convex function is $(1, \gamma)$-strongly qua\-sar-convex ($\gamma \geq 0$). \cris{Similarly, a function is $(1, \gamma)$-quasar-convex if and only if it is strongly quasiconvex (see \cite{necoara2019linear}).} These notions are variants of convexity and strong convexity, respectively, where the conditions are required to hold only at a minimizer $x^{*}$, rather than for all points in $\RR^{p}$. 
We recall a characterization of quasar-convexity.

\medskip

\begin{lemma}\cite[Lemma 10]{hinder2020near}\label{L2}
Let $f \colon \XX \to \RR$ be continuously differentiable  function with minimizer $x^{*} \in \argmin\limits_{\XX}\,f$, where the domain $\XX \subseteq \RR^{p}$ is open and convex. Then, $f$ is quasar-convex with modulus $\kappa \in ]0, 1]$ if and only if
\begin{align*}
f(\lambda x^{*}+(1-\lambda)x)  \leq \kappa \lambda f(x^{*})+(1-\kappa \lambda)f(x), ~ \forall ~ x \in \XX, ~ \forall ~ \lambda \in [0, 1].
\end{align*}
\end{lemma}
\medskip

\begin{proposition}\label{prop1}\cris{
Let $\ft \colon \RR \to \RR$ be a $\kappa$-quasar-convex function with respect to the point $x=0$ on $\RR$, such that $\min_{x \in \RR} \ft(x) = \ft(0) = 0,$ and assume $\ft \in \mathcal{C}^{1}$. Let $\fh \colon \mathcal{S}^{p-1} \to \RR$ be a nonnegative function that is continuously differentiable on the unit sphere $\mathcal{S}^{p-1} \subset \RR^{p}$. We define the function $f$ as $f(x)=\ft(\|x\|) \fh\left(\frac{x}{\|x\|}\right)$ with $f(0)=0.$
Then $f$ is nonnegative, of class $\mathcal{C}^{1}$ on $\RR^{p}\setminus\{0\}$, and $\kappa$-quasar-convex on $\RR^{p}$ with respect to $x^{*}=0$. 
Moreover, for every $x \in \RR^{p}\setminus\{0\},$ the gradient of $f$ is given by:}
\begin{align*}
\nabla f(x)=\frac{x}{\|x\|} \ft'(\|x\|) \fh\left(\frac{x}{\|x\|}\right)+\frac{1}{\|x\|} \left(\Id-\frac{xx^\top}{\|x\|^{2}}\right) \ft(\|x\|) \nabla \fh\left(\frac{x}{\|x\|}\right).
\end{align*}
\end{proposition}

\begin{proof}
Applying the product rule component-wise, we obtain:
\begin{align*}
\nabla f(x)=&\nabla_x \ft(\|x\|) \fh\left(\frac{x}{\|x\|}\right)+\ft(\|x\|) \nabla_x \fh\left(\frac{x}{\|x\|}\right)\\=&\frac{x}{\|x\|} \ft'(\|x\|) \fh\left(\frac{x}{\|x\|}\right)+\frac{1}{\|x\|} \left(\Id-\frac{xx^\top}{\|x\|^{2}}\right) \ft(\|x\|) \nabla \fh\left(\frac{x}{\|x\|}\right).
\end{align*}
Since both $\ft$ and $\fh$ are $\mathcal{C}_1$, it follows that $f(x)$ is also $\mathcal{C}_1$ on $\RR^{p}\setminus\{0\}$. The non-negativity of $\ft$ and $\fh$ implies that $f$ is nonnegative as well. The quasar-convexity of $f$ follows from \cite[Proposition 8]{hermant2024study}.
\hfill
\end{proof}

\medskip

\begin{example}\label{ex1} 
Consider the function $f \colon \RR^{2} \rightarrow \RR$, defined as: 
\begin{align*} 
f(x)=\ft(\|x\|) \fh \left( \frac{x}{\|x\|} \right), 
\end{align*} 
where $\ft(t) = \frac{t^{2}}{1+t^{2}}$ and $\fh(x_1, x_2)=\frac{1}{4N} \sum\limits_{i=1}^{N} \left( a_{i} \sin(b_{i} x_1)^{2}+c_{i} \cos(d_{i} x_2)^{2} \right)$,
with $N=10$. Here, ${a_{i}}$ and ${c_{i}}$ are independently and uniformly 
distributed on $[0, 20]$, while ${b_{i}}$ and ${d_{i}}$ are independently and 
uniformly distributed on $[-25, 25]$. The function is illustrated on the left side of 
Figure \ref{fig1}. 
\medskip

\noindent \cris{Since $\ft$ is $\tfrac{2}{1+R^{2}}$-quasar-convex on $[-R,R]$, by Proposition \ref{prop1}  (see Example \ref{Example 5}), the function $f$ is also $\tfrac{1}{1+R^{2}}$-quasar-convex on the ball centered at the origin with radius $R$, for any $\fh$ that satisfies the assumptions of Proposition \ref{prop1}. Note that this construction is similar to the one introduced in \cite{hermant2024study}, with the main difference being that here $\ft$ is only quasar-convex and $\fh$ is nonnegative, while in their case $\hat{\ft}_{1}(t)=t^{2}$ is strongly convex and $\hat{\fh}_{2}=\fh+1$ is lower bounded by $1$. Moreover, for any fixed direction, meaning that $c=\tfrac{x}{\|x\|}$ is constant, the function behaves like $\fh(c)\cdot \ft(\|x\|)$. However, since $\fh$ is a linear combination of high-frequency trigonometric functions, it can exhibit highly non-convex behavior. This is evident when examining the level set of a segment between two arbitrary points $u$ and $v$ that does not contain the minimizer, as shown on the left side of Figure \ref{fig1}.}

\end{example}

\begin{figure}[htbp] \centering \includegraphics[width=0.9\linewidth]{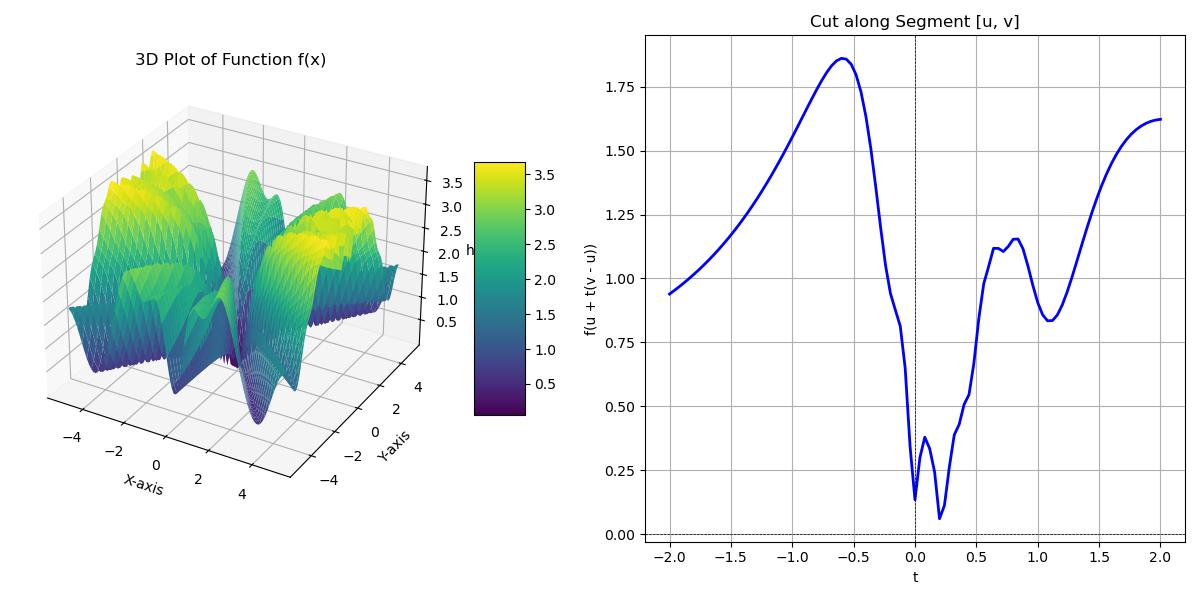} \caption{An illustration of the non-convex function $f(x)$ described in Example \ref{ex1}. On the left-hand side, a 3D plot of the function $f(x)$ is displayed, while on the right-hand side, an arbitrary segment that does not contain the minimizer is shown, highlighting the non-convexity of $f$.}\label{fig1} \end{figure}
\medskip

\noindent A differentiable function $f\colon \RR^{p} \to \RR$ is said to be (\cite[Definition 1]{hinder2020near}) $\Gamma$-weakly smooth with respect to $x^{*} \in \argmin\limits_{x\in\RR^{p}} f(x),$ if for any $x \in \RR^{p}$, we have 
$$\|\nabla f(x)\|^{2} \leq \Gamma \left(f(x)-f(x^{*})\right).$$

\noindent \cris{For every positive integer $\ell > 0$, we define $[\ell] := \{1, \ldots, \ell\}$.} The following lemma provides a method to interchange the limits of a double sum, and will be applied in the proof of Theorem \ref{theo:02}.

\medskip
\begin{lemma}\label{L:change}
Let $d,T\in\mathbb{N}$ with $d\ge 1$, and let $\{a_k\}_{k=1}^{T}$ be a sequence of nonnegative numbers. Then
\begin{equation}\label{eq:change}
\sum\limits_{t=1}^{T-1} \sum\limits_{k=t+1}^{\min\{T,\,t+d-1\}} a_k
=\sum\limits_{k=2}^{T} \sum\limits_{t=\max\{1,\,k-d+1\}}^{k-1} a_k
\le (d-1)\sum\limits_{k=1}^{T} a_k.
\end{equation}
\end{lemma}
\begin{proof}\cris{
The equality is obtained by changing the order of summation. Indeed, the left-hand side sums $a_k$ over all pairs $(t,k)$ satisfying $1\le t\le T-1,$ $t+1\le k\le \min\{T,t+d-1\}.$ For a fixed $k\in\{2,\dots,T\}$, $t$ are satisfies that $k-d+1\le t\le k-1$ and $1\le t\le T-1$, which is equivalent to
$\max\{1,k-d+1\}\le t\le k-1$, which are the limits of the equality in \eqref{eq:change}.\\}

\noindent \cris{For the inequality, note that for each fixed $k$ the inner sum $\sum\limits_{t=\max\{1,k-d+1\}}^{k-1} a_k$ contains at most $d-1$ constant terms. Hence
$$\sum\limits_{t=\max\{1,k-d+1\}}^{k-1} a_k \le (d-1)\,a_k,$$
and summing this bound over $k=2,\ldots,T$  gives the inequality in \eqref{eq:change}.\hfill}
\end{proof}

\section{Main Results}\label{sec:mr}

\noindent In this section, inspired by the techniques developed in \cite{wan2022online}, we derive non-stationary regret bounds for the iterates generated by DOGD when applied to quasar-convex functions. In each round $t$, DOGD queries the gradient $\nabla f_{t}(x_{t})$, but due to the delay, this gradient is received at the end of round $t'=t+d_{t}-1$ and can be utilized in round $t+d_{t}$, where $d_{t} \in \mathbb{Z}^{+}$ represents a nonnegative integer delay.  DOGD updates the decision $x_{t}$ with the sum of the gradients received by the end of round $t$. The set of rounds whose gradients are received at each round $t$ is denoted by $\FF_{t}=\{ k \in [T] \mid k+d_{k}-1=t \}$, where $[T]:=\{1,\ldots, T\}$. In the setting without delays, for all $t$, $D=T$, $d=d_{t}=1$, $t'=t$, and $\FF_{t}=\{t\}$. The detailed procedure for DOGD is summarized in Algorithm \ref{DOGD-SC}. Specifically, we address cases where the functions are $L$-Lipschitz and $\Gamma$-weakly smooth. 
\begin{algorithm}[ht!]
		\caption{Delayed Projected Online Gradient Descent for quasar-convex functions}
		\begin{algorithmic}[1]
			\State \textbf{Initialization:} Choose an arbitrary vector $x_{1} \in \XX$ and a sequence of step-sizes $\left\{\eta_{t}\right\}_{t=1}^{T}$
			\For{$t=1, 2, \ldots, T$}
			\State Query $r_{t}\approx \nabla f_{t}(x_{t}).$ 
			\State $x_{t+1}=\begin{cases}
				P_{\XX} \left( x_{t}-\eta_{t} \sum\limits_{k \in \FF_{t}} r_{k} \right) & \text{if } |\FF_{t}| > 0 \\
				x_{t} & \text{otherwise}
			\end{cases}$
			\EndFor
		\end{algorithmic}
		\label{DOGD-SC}
\end{algorithm}

\subsection{DOGD for Quasar-Convex Functions with Non-Stationary Regret}	

In the quasar-convex case, the following assumptions will be needed:
\medskip
\begin{assumption}
\begin{enumerate}[label=(B\arabic*)] \
\item\label{B4} The radius of the convex decision set $\XX$ is bounded by $R$, i.e., $\|x\| \leq R$ for all $x \in \XX$.
\item\label{B3} Each loss function $f_{t}$ is $\kappa$-quasar-convex for some $x^{*}_{t}\in \argmin\limits_{x\in \XX}f_{t}(x)$.
\item\label{B2} Each loss function $f_{t} \colon \XX \to \RR$ is $L$-Lipschitz.
\item\label{B5} Each loss function $f_{t} \colon \XX \to \RR$ is $\Gamma$-weakly smooth.
\end{enumerate}
\end{assumption}

\noindent To simplify the notation, for every $t \in [T]$, we define $r_{t}=\nabla f_{t}(x_{t})+m_{t}$, where $m_{t} \in \RR^{p}$ represents the deterministic gradient error. We assume that $m_{t}$ is bounded for some constant $\delta_{t} > 0$. The cumulative gradient error and the cumulative squared gradient over the horizon $T$ are defined as:
\begin{align}
\Delta_{T}=\sum\limits_{t=1}^{T} \delta_{t} \hspace{2mm} \text{and} \hspace{2mm} \Lambda_{T}=\sum\limits_{t=1}^{T} \delta^{2}_{t},\nonumber
\end{align}
respectively.
Additionally, we define 
$$x'_{t+1}=\begin{cases}
		x_{t}-\eta_{t}\sum\limits_{k \in \FF_{t}} r_{k}
  & \text{if } |\FF_{t}| > 0 \\
		x_{t} & \text{otherwise}
	\end{cases}$$
	which the term inside of the projection. Then, by definition, for every $t\in [T+d-1],$ it holds that
	\begin{align}\label{iterates}
\eta_{t}\sum\limits_{k\in\FF_{t}}r_{k}=(x_{t}-x'_{t+1}).
	\end{align}
 		According to Algorithm \ref{DOGD-SC}, there could exist some feedback that arrives after the round $T$ and is not used to update the decision. However, it is useful for the analysis. Therefore, for $t \in [T+1, T+d-1]$, we also define $\FF_{t}=\{k\in [T] \mid k+d_{k}-1=t\}$, and  the virtual update as:
		$$
		 x_{t+1}=\begin{cases}
			P_\XX \left( x_{t}-\eta_{t} \sum\limits_{k \in \FF_{t}}r_{k} \right) & \text{if } |\FF_{t}| > 0, \\
			x_{t} & \text{otherwise}.
		\end{cases}
		$$
Analogously, for any $t \in [T+1,T+d-1]$, we also 
 define
		$$
		x'_{t+1}=\begin{cases}
			x_{t}-\eta_{t} \sum\limits_{k \in \FF_{t}} r_{k} & \text{if } |\FF_{t}| > 0, \\
			x_{t} & \text{otherwise}.
		\end{cases} 
		$$
\medskip Let us also define $s=\min \left\{ t \in [T+d-1] \mid |\FF_{t}| > 0 \right\}.$ \cris{Observe that $\bigcup_{t=s}^{T+d-1} \FF_{t}=\bigcup_{t=1}^{T+d-1} \FF_{t}$ and by definition of $\FF_{t}$, $\bigcup_{t=s}^{T+d-1} \FF_{t}\subset [T]$. Since all feedback arrives between rounds $s$ up to $T+d-1$, we have that $\bigcup_{t=s}^{T+d-1} \FF_{t}= [T]$. Furthermore, since each feedback is received only once, $\FF_{i} \cap \FF_j=\emptyset$ for any $i \neq j$. Combining these results, for every finite sequence $\left\{a_{k}\right\}_{k=1}^{T}$ of real numbers, we have:\begin{align}\label{sets}
\sum\limits_{t=1}^{T} a_{t}=\sum\limits_{t=s}^{T+d-1} \sum\limits_{k \in \FF_{t}} a_{k}.
\end{align}}

\noindent  Our main result is given below.
\medskip

\begin{theorem}\label{theo:02}
Under assumptions \ref{B4}-\ref{B3}, the following assertions hold:
 \begin{enumerate}
     \item\label{T:Lipschitz} If \ref{B2} holds, the stationary regret of Algorithm \ref{DOGD-SC} with a constant step-size $0<\eta_{t}=\eta$ can be upper bounded by:
     \begin{align}\label{eq:R1} R^{NS}_{T}\leq &  \frac{2R^{2}}{\eta\kappa}+\frac{3R+\eta L(d-1)}{\eta\kappa}V_{T}+  \frac{L^{2}d+4(d-1)L^{2}}{2\kappa}\eta T+ \frac{\eta d}{2\kappa}\Lambda_{T}+\frac{\eta dL+2R+2\eta(d-1)L}{\kappa}\Delta_{T} \end{align}
      \item\label{T:Weakly-smooth} If \ref{B5} holds, the stationary regret of Algorithm \ref{DOGD-SC} with a constant step-size $0<\eta_{t}=\eta< \frac{1}{\bar{\alpha}}=\frac{2\kappa}{\left(d+4d^{\frac{1}{2}}(d-1)\right)\Gamma}$ can be upper bounded by:
\begin{align}\label{eq:R2}
R^{NS}_{T} \leq \frac{b^{2}+2ac+b \left(b^{2}+4ac\right)^{\frac{1}{2}}}{2a^{2}},
\end{align}
where $a=1-\alpha\eta>0$, 
$$b= \frac{\left(2R(d-1)\Gamma V_{T}\right)^{\frac{1}{2}}}{\kappa}+\frac{(2d)^{\frac{1}{2}}(d-1)\eta+\eta d}{\kappa}\left(\Gamma \Lambda_{T}\right)^{\frac{1}{2}},\hspace{2mm}\text{ and }\hspace{2mm}c=\frac{2R^{2}}{\kappa}+\frac{3RV_{T}}{\kappa}+ \frac{\eta d}{2\kappa}\Lambda_{T}+ \frac{2R}{\kappa} \Delta_{T}.$$
\end{enumerate}
\end{theorem}

\begin{proof}
{\color{black} For every positive integer $t$, let $\bar{t} := \min\{t, T\}$.
 By Assumption \ref{B3}, for every $t \in [T]$, we have that:
 \begin{align}\label{eq:b1}
 \kappa\left( f_{t}(x_{t})-f_{t}(x^{*}_{t})\right) \leq & \scal{\nabla f_{t}(x_{t})}{
  x_{t}-x^{*}_{t}} \nonumber \\ 
  = & \left( \scal{\nabla f_{t}(x_{t})}{x_{t}-x_{t'}}
 + \scal{\nabla f_{t}(x_{t})}{x_{t'}  -x^{*}_{\bar{t'}}} + \scal{\nabla f_{t}(x_{t})}{x^{*}_{\bar{t'}}-x^{*}_{t}} \right) \nonumber\\ 
  =&  \scal{\nabla f_{t}(x_{t})}{x_{t}-x_{t'}}+ \scal{ r_{t}}{x_{t'}-x^{*}_{\bar{t'}}}- \scal{ m_{t}}{x_{t'}-x^{*}_{\bar{t'}}}+\scal{\nabla f_{t}(x_{t})}{x^{*}_{\bar{t'}}-x^{*}_{t}}.
 \end{align} 
 \noindent Summing \eqref{eq:b1} from $1$ to $T$, we obtain:
\begin{align}
 \kappa R_{T}^{NS} \leq  &  \sum\limits_{t=1}^{T} \scal{\nabla f_{t}(x_{t})}{x_{t}-x_{t'}} +\sum\limits_{t=1}^{T}\left(\scal{ r_{t}}{x_{t'} 
 - x^{*}_{\bar{t'}}}-\scal{m_{t}}{x_{t'}-x^{*}_{\bar{t'}}}\right) +
 \sum\limits_{t=1}^{T} \scal{\nabla f_{t}(x_{t})} 
 {x^{*}_{\bar{t'}} - x^{*}_{t}} & \nonumber\\ 
 \leq & \sum\limits_{t=1}^{T} \scal{\nabla f_{t}(x_{t})}{x_{t}-x_{t'}} +
 \sum\limits_{t=1}^{T}\left(\scal{r_{t}}{x_{t'}-x^{*}_{\bar{t'}}}+\|m_{t}|\| 
 \|x_{t'}-x^{*}_{\bar{t'}}\|\right) + \sum\limits_{t=s}^{T+d-1}
 \sum\limits_{k \in \FF_{t}}  \scal{\nabla f_{k}(x_{k})}{x^{*}_{\bar{t}}-x^{*}_{k}} \nonumber \\ 
 \leq & \sum\limits_{t=1}^{T} \scal{\nabla f_{t}(x_{t})}{x_{t}-x_{t'}} + 
 \sum\limits_{t=1}^{T} \scal{r_{t}}{x_{t'} - x^{*}_{\bar{t'}}} +  2R \Delta_{T} 
 + \sum\limits_{t=s}^{T+d-1} \sum\limits_{k\in\FF_{t}} \|\nabla f_{k}(x_{k})\| \| 
 x^{*}_{\t}-x^{*}_{k}\|. \label{eq:b2}
 \end{align}
 \cris{where the second inequality follows from Cauchy–Schwarz and \eqref{sets}, while the third inequality follows from the definition of $\Delta_{T}$, Assumption \ref{B4}, and Cauchy–Schwarz.}\medskip
 
\noindent Now, we establish an upper bound for the second term on the right side of \eqref{eq:b2}:
\begin{align}\label{eq:c6}
&\sum\limits_{t=1}^{T} \scal{r_{t}}{ x_{t'}-x^{*}_{\bar{t'}}} 
= \sum\limits_{t=s}^{T+d-1} \sum\limits_{k\in\FF_{t}}\scal{r_{k}}{x_{t}-x^{*}_{\t}}= \sum\limits_{t=s}^{T+d-1} \scal{\sum\limits_{k\in\FF_{t}}r_{k}}{x_{t}-x^{*}_{\t}}= \sum\limits_{t=s}^{T+d-1} \frac{1}{\eta} \scal{x_{t}-x_{t+1}'}{x_{t}-x^{*}_{\t}}\nonumber\\
= & \sum\limits_{t=s}^{T+d-1} \frac{1}{2\eta} \left( \|x_{t}-x^{*}_{\t}\|^{2}-\|x_{t+1}'-x^{*}_{\t}\|^{2}+\|x_{t}-x_{t+1}'\|^{2} \right) \nonumber \\
 =& \frac{1}{2\eta}\sum\limits_{t=s}^{T+d-1}  \left( \|x_{t}-x^{*}_{\t}\|^{2}-\|x_{t+1}'-x^{*}_{\t}\|^{2}+\eta^{2}\left\| \sum\limits_{k \in \FF_{t}}r_{k} \right\|^{2}  \right) \nonumber \\
\leq& \frac{1}{2\eta}\sum\limits_{t=s}^{T+d-1}  \left( \|x_{t}-x^{*}_{\t}\|^{2}-\|x_{t+1}-x^{*}_{\t}\|^{2}+\eta^{2}\left\| \sum\limits_{k \in \FF_{t}}r_{k} \right\|^{2} \right)  \nonumber \\
\leq& \frac{1}{2\eta}\sum\limits_{t=s}^{T+d-1}  \left( \|x_{t}-x^{*}_{\t}\|^{2}-\|x_{t+1}-x^{*}_{\t}\|^{2}+\eta^{2}\sum\limits_{k \in \FF_{t}} |\FF_{t}|\|r_{k} \|^{2} \right)  \nonumber \\
=& \frac{1}{2\eta} \sum\limits_{t=s}^{T-1} \left( \|x_{t}-x^{*}_{\t}\|^{2}-\|x_{t+1}-x^{*}_{\t}\|^{2}\right)+\frac{1}{2\eta} \sum\limits_{t=T}^{T+d-1}  \left( \|x_{t}-x^{*}_{T}\|^{2}-\|x_{t+1}-x^{*}_{T}\|^{2}\right)+ \frac{\eta}{2}\sum\limits_{t=s}^{T+d-1} \sum\limits_{k \in \FF_{t}} |\FF_{t}|\|r_{k} \|^{2}\nonumber \\ 
=& \frac{1}{2\eta}\sum\limits_{t=s}^{T-1}  \left( \|x_{t}-x^{*}_{t}\|^{2}-\|x_{t+1}-x^{*}_{t}\|^{2}\right)+\frac{1}{2\eta}\left( \|x_{T}-x^{*}_{T}\|^{2}-\|x_{T+d}-x^{*}_{T}\|^{2}\right)+ \frac{\eta}{2}\sum\limits_{t=s}^{T+d-1} \sum\limits_{k \in \FF_{t}} |\FF_{t}|\|r_{k} \|^{2}\nonumber \\ 
\leq & \frac{1}{2\eta}\sum\limits_{t=s}^{T-1} \left( \|x_{t}-x^{*}_{t}\|^{2}-\|x_{t+1}-x^{*}_{t}\|^{2}\right) +\frac{1}{2\eta}\left( \|x_{T}-x^{*}_{T}\|^{2}-\|x_{T+d}-x^{*}_{T}\|^{2}\right)+\frac{\eta d}{2}\sum\limits_{t=s}^{T+d-1} \sum\limits_{k \in \FF_{t}} \|r_{k} \|^{2},
\end{align}
\cris{where  the first equality follows from \eqref{sets}, the  third  equality and the fifth equalities follow from \eqref{iterates}. The first inequality is due to the non-expansiveness of the projection operator, while the second inequality results from the convexity of the norm. The sixth equality follows by splitting the sum into two parts, from $s$ to $T-1$ and from $T$ to $T+d-1$, and noting that in the second sum we have $\bar{t}=T$. The final equality holds developing the second term which is a telescopic sum.  The final inequality holds in view of, for every $k \in [T+d]$, $|\FF_{k}|\leq d$.}\medskip

\noindent We now proceed to develop the last sum in \eqref{eq:c6}. It follows from \eqref{sets} and the definition of $r_{t}$ that
\begin{align}\label{eq:c6.5}
    \sum\limits_{t=s}^{T+d-1} \sum\limits_{k \in \FF_{t}} \|r_{k} \|^{2}=\sum\limits_{t=1}^{T}  \|r_{t}\|^{2}=\sum\limits_{t=1}^{T} \|\nabla f_{t}(x_{t})+m_{t} \|^{2} 
    \leq&\sum\limits_{t=1}^{T} \left(\|\nabla f_{t}(x_{t})\|+\|m_{t}\|\right)^{2}\nonumber\\ 
    =&\sum\limits_{t=1}^{T} \left(\|\nabla f_{t}(x_{t})\|^{2}+2\|\nabla f_{t}(x_{t})\|\|m_{t}\|+\|m_{t}\|^{2}\right)\nonumber\\ 
    \leq&\sum\limits_{t=1}^{T} \left(\|\nabla f_{t}(x_{t})\|^{2}+2\|\nabla f_{t}(x_{t})\|\delta_{t}+\delta^{2}_{t}\right)\nonumber\\ 
    \leq&\sum\limits_{t=1}^{T} \left(\|\nabla f_{t}(x_{t})\|^{2}+2\|\nabla f_{t}(x_{t})\|\delta_{t}\right)+\Lambda_{T},
\end{align}
\cris{where the second inequality follows from the boundedness of $m_{t}.$ While the third is due to the definition of $\Lambda_{T}.$}  On the other hand, \cris{by triangle inequality and Assumption \ref{B4}}, for every $t\in [T-1]$, we have that:
\begin{align*} \|x_{t+1}-x^{*}_{t+1}\|^{2}\leq &\|x_{t+1}-x^{*}_{t}\|^{2}+2\|x_{t+1}-x^{*}_{t}\|\|x^{*}_{t+1}-x_{t}^{*}\|+\|x^{*}_{t+1}-x_{t}^{*}\|^{2}\nonumber\\ \leq &\|x_{t+1}-x^{*}_{t}\|^{2}+\left(2\|x_{t+1}-x^{*}_{t}\|+\|x^{*}_{t+1}-x_{t}^{*}\|\right)\|x^{*}_{t+1}-x_{t}^{*}\|\nonumber\\ \leq &\|x_{t+1}-x^{*}_{t}\|^{2}+6R\|x^{*}_{t+1}-x_{t}^{*}\|,
\end{align*}
which implies that
\begin{align}\label{eq:c7} -\|x_{t+1}-x^{*}_{t}\|^{2}\leq &-\|x_{t+1}-x^{*}_{t+1}\|^{2}+6R\|x^{*}_{t+1}-x_{t}^{*}\|.
\end{align}
\noindent Plugging \eqref{eq:c6.5} and \eqref{eq:c7} in \eqref{eq:c6} yields:
\begin{align}\label{eq:c8}
  \sum\limits_{t=1}^{T} \scal{r_{k}}{x_{t'}-x^{*}_{t'}} 
 \leq & \frac{1}{2\eta} \sum\limits_{t=s}^{T-1} \left( \|x_{t}-x^{*}_{t}\|^{2} -\|
 x_{t+1}-x^{*}_{t+1}\|^{2} + 6R\|x^{*}_{t+1}-x_{t}^{*}\|\right)  \nonumber\\ 
 &+\frac{1}{2\eta}\left( \|x_{T} - x^{*}_{T}\|^{2} - \|x_{T+d} - x^{*}_{T}\|^{2}
 \right)  + \frac{\eta d}{2}\sum\limits_{t=1}^{T} \left(\|\nabla f_{t}(x_{t})\|^{2} 
 + 2 \|\nabla f_{t}(x_{t})\| \delta_{t} \right)+\frac{\eta d}{2}
 \Lambda_{T} \nonumber\\ 
 = & \frac{1}{2\eta} \sum\limits_{t=s}^{T-1} \left( \|x_{t}-x^{*}_{t}\|^{2} -\|
 x_{t+1}-x^{*}_{t+1}\|^{2}\right) + \frac{3R}{\eta} \sum\limits_{t=s}^{T-1}\|x^{*}_{t+1}-x_{t}^{*}\|  \nonumber\\ 
 &+\frac{1}{2\eta}\left( \|x_{T} - x^{*}_{T}\|^{2} - \|x_{T+d} - x^{*}_{T}\|^{2}
 \right)  + \frac{\eta d}{2}\sum\limits_{t=1}^{T} \left(\|\nabla f_{t}(x_{t})\|^{2} 
 + 2 \|\nabla f_{t}(x_{t})\| \delta_{t} \right)+\frac{\eta d}{2}
 \Lambda_{T} \nonumber\\ 
 = & \frac{1}{2\eta}\left( \|x_{s}-x^{*}_{s}\|^{2} -\|
 x_{T}-x^{*}_{T}\|^{2}\right) + \frac{3R}{\eta} V_{T} \nonumber\\ 
 &+\frac{1}{2\eta}\left( \|x_{T} - x^{*}_{T}\|^{2} - \|x_{T+d} - x^{*}_{T}\|^{2}
 \right)  + \frac{\eta d}{2}\sum\limits_{t=1}^{T} \left(\|\nabla f_{t}(x_{t})\|^{2} 
 + 2 \|\nabla f_{t}(x_{t})\| \delta_{t} \right)+\frac{\eta d}{2}
 \Lambda_{T} \nonumber\\ 
 \leq & \frac{1}{2\eta} \|x_{s}-x^{*}_{s}\|^{2}+ \frac{3R}{\eta} 
 V_{T}+\frac{\eta d}{2}
 \sum\limits_{t=1}^{T} \left(\|\nabla f_{t}(x_{t})\|^{2}+2\|\nabla f_{t}(x_{t})\| 
 \delta_{t}\right) + \frac{\eta d}{2} \Lambda_{T} \nonumber\\ 
 \leq & \frac{2R^{2}}{\eta}+ \frac{3R}{\eta}V_{T}+\frac{\eta d}{2} 
 \sum\limits_{t=1}^{T} \|\nabla f_{t}(x_{t})\|^{2}+\eta d
 \sum\limits_{t=1}^{T} \|\nabla f_{t}(x_{t})\| 
 \delta_{t} + \frac{\eta d}{2} \Lambda_{T},
\end{align}
\cris{where the second equality is obtained by developing the telescopic sum and using the definition of $V_{T}$, while the last inequality follows from assumption \ref{B4}.} \medskip

\noindent \cris{Now we proceed to bound the first and third terms of \eqref{eq:b2} by considering different cases.}

\begin{enumerate}
\item \textbf{Case  $f_{t}$ is $L$-Lipschitz}:
We now proceed to upper bound the last term in the right-hand side of \eqref{eq:b2}:
\cris{\begin{align}\label{eq:c5}
& \sum\limits_{t=s}^{T+d-1} \sum\limits_{k\in\FF_{t}} \|\nabla f_{k}(x_{k})\| \|x^{*}_{\t}-x^{*}_{k}\| \leq L\sum\limits_{t=s}^{T+d-1} \sum\limits_{k\in\FF_{t}} \|x^{*}_{\t}-x^{*}_{k}\|=L \sum\limits_{t=1}^{T}\|x^{*}_{\overline{t'}}-x^{*}_{t}\|\nonumber\\
= &L \sum\limits_{t=1}^{T-1}\|x^{*}_{\overline{t'}}-x^{*}_{t}\| \leq L \sum\limits_{t=1}^{T -1} \sum\limits_{k=t+1}^{\overline{t'}}\|x^{*}_{k}-x^{*}_{k-1}\| 
= L \sum\limits_{t=1}^{T-1}\sum\limits_{k=t+1}^{\overline{t+d_{t}-1}}\|x^{*}_{k}-x^{*}_{k-1}\| 
 \nonumber\\ \leq & L \sum\limits_{t=1}^{T-1} \sum\limits_{k=t+1}^{\overline{t+d-1}}\|x^{*}_{k}-x^{*}_{k-1}\|\leq  L(d-1)\sum\limits_{t=1}^{T-1}\|x^{*}_{t+1}-x^{*}_{t}\|=L (d-1)V_{T},
  \end{align}}
\cris{where the first inequality follows from Assumption \ref{B2}, the first equality is due to \eqref{sets}, and the second equality is due to $x^{*}_{\overline{T'}}=x^{*}_{T}$. The second inequality comes from the triangle inequality. The third inequality is given by $d=\max\limits_{t \in [T]} d_{t}$, while the fourth follows from  Lemma \ref{L:change}.} \medskip

\noindent Substituting \eqref{eq:c8} and \eqref{eq:c5} into \eqref{eq:b2} and rearranging terms, we obtain:
  \begin{align}
  \kappa R^{NS}_{T}  \leq &  \sum\limits_{t=1}^{T} \scal{\nabla f_{t}(x_{t})}{x_{t}-x_{t'}}+ \frac{2R^{2}}{\eta}+ \frac{3R}{\eta}V_{T}+\frac{\eta d}{2}\sum\limits_{t=1}^{T} \left(\|\nabla f_{t}(x_{t})\|^{2}+2\|\nabla f_{t}(x_{t})\|\delta_{t}\right)\nonumber \\ &+\frac{\eta d}{2} \Lambda_{T}+2R\Delta_{T}+ L(d-1)V_{T}  \nonumber\\ \leq &  \sum\limits_{t=1}^{T} \|\nabla f_{t}(x_{t})\|\|x_{t}-x_{t'}\|+  \frac{2R^{2}}{\eta}+ \frac{3R}{\eta}V_{T}+\frac{\eta d}{2}\sum\limits_{t=1}^{T} \left(L^{2}+2L\delta_{t}\right)+\frac{\eta d}{2} \Lambda_{T}+2R\Delta_{T}+ L(d-1)V_{T}  \nonumber\\  \leq & \sum\limits_{t=1}^{T}L\|x_{t}-x_{t'}\|+\frac{2R^{2}}{\eta}+\frac{3R}{\eta}V_{T}+\frac{\eta dTL^{2}}{2}+\left(\eta dL+2R\right)\Delta_{T}+ \frac{\eta d}{2}\Lambda_{T}+L(d-1)V_{T},
   \label{eq:c9}
 \end{align}
 \cris{where the second inequality follows from the Cauchy-Schwarz inequality and Assumption \ref{B4}. The third inequality is derived from Assumption \ref{B4} and the development of  the second sum.} Then, it remains to bound $\sum\limits_{t=1}^{T} L\|x_{t}-x_{t'}\|$. To do so, we use the triangle inequality to obtain:
\begin{align}
\|x_{t'}-x_{t}\| \leq& \sum\limits_{i=t}^{t'-1} \|x_{i+1}-x_{i}\| \leq \sum\limits_{i=t}^{t'-1} \|x_{i+1}'-x_{i}\| =\sum\limits_{i=\max(t,s)}^{t'-1} \|x_{i+1}'-x_{i}\|= \eta\sum\limits_{i=\max(t,s)}^{t'-1} \left\| \sum\limits_{k \in \FF_{i}} r_{k}\right\|\nonumber\\ =& \eta\sum\limits_{i=\max(t,s)}^{t'-1} \left\| \sum\limits_{k \in \FF_{i}}\left(\nabla f_{k}(x_{k})+m_{k}\right)\right\| \leq  \eta\sum\limits_{i=\max(t,s)}^{t'-1} \left\| \sum\limits_{k \in \FF_{i}}\nabla f_{k}(x_{k})\right\|+\eta\sum\limits_{i=\max(t,s)}^{t'-1} \left\| \sum\limits_{k \in \FF_{i}}m_{k}\right\|\nonumber\\ \leq& \eta\sum\limits_{i=\max(t,s)}^{t'-1}  |\FF_{i}|L+\eta \sum\limits_{i=\max(t,s)}^{t'-1} \sum\limits_{k \in \FF_{i}}\|m_{k}\|\label{eq:c10}
\end{align}
\cris{where the first inequality follows from the triangle inequality, the second inequality from the non-expansiveness of the projection, the first equality from the fact that $x_{k+1}'=x_{k}$ for $k\leq s-1$, the second equality from \eqref{iterates}, and the third inequality from the fact that, for every $k \in [T+d]$, $|\FF_{k}|\leq d$.} Summing \eqref{eq:c10} from $1$ to $T$ we have that 
\cris{\begin{align}\label{eq:c11}
\sum\limits_{t=1}^{T}\|x_{t'}-x_{t}\| \leq & \eta \sum\limits_{t=1}^{T}\sum\limits_{i=\max(t,s)}^{t'-1} |\FF_{i}| L+\eta\sum\limits_{t=1}^{T}\sum\limits_{i=\max(t,s)}^{t'-1} \sum\limits_{k \in \FF_{i}}\delta_{k}
\end{align}}

To bound these terms, first note that $s \leq d$. Indeed, if $s > d$, then for every $k \in \FF_{s}$, $k+d_{k}-1=s > d$, implying $1 \notin \FF_{s}$. By the definition of $s$, $d_{1}=1+d_{1}-1 > s > d$, contradicting the maximality of $d$. We then consider two cases:  \medskip

\begin{enumerate}
        \item \textbf{Case $s>T$:} In this case, the algorithm does not update, and we have:
              \begin{align*}
\sum\limits_{t=1}^{T} \sum\limits_{i=\max(t,s)}^{t'-1} |\FF_{i}| &= \sum\limits_{t=1}^{T} \sum\limits_{i=s}^{t'-1} |\FF_{i}|=\sum\limits_{t=1}^{T} \sum\limits_{i=s}^{t+d_{t}-2} |\FF_{i}| \leq \sum\limits_{t=1}^{T} \sum\limits_{i=s}^{T+d-2} |\FF_{i}|\leq (d-1) \sum\limits_{t=s}^{T+d-1} |\FF_{i}| = (d-1)T,
\end{align*}

where the first equality holds in view of $s \geq T$, ensuring $\max(t,s) \geq t$ for all $t \in [T]$. The second equality and first inequality follow from $t'=t+d_{t}-1 \leq T+d-1$, while the second inequality holds since $d \geq s > T$. Similarly,

\begin{align*}
\sum\limits_{t=1}^{T} \sum\limits_{i=\max(t,s)}^{t'-1} \sum\limits_{k \in \FF_{i}}\|m_{k}\| &= \sum\limits_{t=1}^{T} \sum\limits_{i=s}^{t'-1} \sum\limits_{k \in \FF_{i}}\|m_{k}\|=\sum\limits_{t=1}^{T} \sum\limits_{i=s}^{t+d_{t}-2} \sum\limits_{k \in \FF_{i}}\|m_{k}\| \leq \sum\limits_{t=1}^{T} \sum\limits_{i=s}^{T+d-2} \sum\limits_{k \in \FF_{i}}\delta_{k} \\
& \leq (d-1) \sum\limits_{t=s}^{T+d-1} \sum\limits_{k \in \FF_{i}}\delta_{k} \leq  (d-1) \Delta_{T}.
\end{align*}

\item \textbf{Case $s \leq T$:} In this case, we have:
\begin{align*}
\sum\limits_{t=1}^{T} \sum\limits_{i=\max(t,s)}^{t'-1} |\FF_{i}| &=\sum\limits_{t=1}^{s-1} \sum\limits_{i=s}^{t'-1} |\FF_{i}|+\sum\limits_{t=s}^{T} \sum\limits_{i=t}^{t'-1} |\FF_{i}| \leq \sum\limits_{t=1}^{s-1} \sum\limits_{i=s}^{T+d-2} |\FF_{i}|+\sum\limits_{t=s}^{T} \sum\limits_{i=t}^{t+d-2} |\FF_{i}| \\
&=\sum\limits_{t=1}^{s-1} \sum\limits_{i=s}^{T+d-2} |\FF_{i}|+\sum\limits_{t=s}^{T+d-2} \sum\limits_{i=\max(s, t-d+2)}^{t}|\FF_{t}| \\
& \leq (s-1) \sum\limits_{i=s}^{T+d-2} |\FF_{i}|+(d-1) \sum\limits_{t=s}^{T+d-2} |\FF_{t}| \leq 2(d-1) \sum\limits_{t=s}^{T+d-2} |\FF_{t}| \leq  2(d-1)T,
\end{align*}
where the first inequality follows from $t'=t+d_{t}-1 \leq t+d-1 \leq T+d-1$, the second equality by reordering the sum, \cris{the second inequality from the fact that the second inner sum has at most $d-1$ constant terms},  and the third inequality since $d \geq s$. Similarly,
\begin{align*}
 & \sum\limits_{t=1}^{T} \sum\limits_{i=\max(t,s)}^{t'-1} \sum\limits_{k \in 
 \FF_{i}}\|m_{k}\|  = \sum\limits_{t=1}^{s-1} \sum\limits_{i=s}^{t'-1} 
 \sum\limits_{k \in \FF_{i}}\|m_{k}\|+\sum\limits_{t=s}^{T} \sum\limits_{i=t}^{t'-1}
 \sum\limits_{k \in \FF_{i}}\|m_{k}\| \\ 
 \leq & \sum\limits_{t=1}^{s-1} \sum\limits_{i=s}^{T+d-2} \sum\limits_{k \in
 \FF_{i}}\delta_{k}+\sum\limits_{t=s}^{T} \sum\limits_{i=t}^{t+d-2} 
 \sum\limits_{k \in \FF_{i}}\delta_{k} = \sum\limits_{t=1}^{s-1} 
 \sum\limits_{i=s}^{T + d-2} \sum\limits_{k \in \FF_{i}} \delta_{k} + \sum\limits_{t=s}^{T 
 + d - 2} \sum\limits_{i=\max(s, t-d+2)}^{t}\sum\limits_{k \in \FF_{t}}\delta_{k} \\
 \leq & (s-1) \sum\limits_{i=s}^{T+d-2} \sum\limits_{k \in \FF_{i}}\delta_{k} + 
 (d-1) \sum\limits_{t=s}^{T+d-2}\delta_{t} \leq 2(d-1) \sum\limits_{t=s}^{T + d - 
 2} \sum\limits_{k \in \FF_{t}}\delta_{k}\leq  2(d-1)\Delta_{T}.
\end{align*}
    \end{enumerate}    

Then, combining the last inequality with \eqref{eq:c9}  and \eqref{eq:c11}, we get:
\begin{align}
  R^{NS}_{T}  \leq &  \frac{2R^{2}}{\eta\kappa}+\frac{3R+\eta L(d-1)}{\eta\kappa}V_{T}+  \frac{L^{2}d+4(d-1)L^{2}}{2\kappa}\eta T+ \frac{\eta d}{2\kappa}\Lambda_{T}+\frac{\eta dL+2R+2\eta(d-1)L}{\kappa}\Delta_{T}. \nonumber
 \end{align}
and the proof of this bound is complete.

\item \textbf{Case  $f_{t}$ is weakly smooth}: We now proceed to upper bound the \cris{last term in the right-hand side of \eqref{eq:b2}}:
\cris{\begin{align}\label{eq:b5}
 &\hspace{-7mm}\left(\sum\limits_{t=s}^{T+d-1} \sum\limits_{k\in\FF_{t}} \|\nabla f_{k}(x_{k})\| \| 
 x^{*}_{\t}-x^{*}_{k}\|\right)^{2}=\left(\sum\limits_{t=1}^{T} \|\nabla f_{t}(x_{t})\|\|x^{*}_{\overline{t'}} -
 x^{*}_{t}\|\right)^{2} \leq \left( \sum\limits_{t=1}^{T}\|x^{*}_{\overline{t'}} -
 x^{*}_{t}\|^{2}\right) \left(\sum\limits_{t=1}^{T} \|\nabla f_{t}(x_{t})\|^{2} 
 \right) \nonumber\\
\leq & 2R \left (\sum\limits_{t=1}^{T}\|x^{*}_{\overline{t'}}-x^{*}_{t}\|\right) \left(\Gamma\sum\limits_{t=1}^{T}\left(f_{t}(x_{t})-f_{t}(x^{*}_{t})\right)\right)
\leq 2R \Gamma \left(\sum\limits_{t=1}^{T}\sum\limits_{k=t+1}^{\overline{t'}}\|x^{*}_{k}-x^{*}_{k-1}\|\right) R^{NS}_{T} \nonumber \\ 
\leq & 2R \Gamma \left( \sum\limits_{t=1}^{T-1}\sum\limits_{k=t+1}^{\overline{t + d 
 - 1}} \|x^{*}_{k}-x^{*}_{k-1}\|\right)  R^{NS}_{T} 
 \leq  2R(d-1) \Gamma \left( \sum\limits_{t=2}^{T} \|x^{*}_{k} - x^{*}_{k-1} \| 
 \right) R^{NS}_{T} = 2R(d-1) \Gamma V_{T} R^{NS}_{T},
 \end{align}}
\noindent\cris{where the first equality is due to \eqref{sets}, the first inequality follows from applying Cauchy-Schwarz, the second from Assumption \ref{B5}, and the third from the triangle inequality and the definition of $R_{T}^{NS}$. The fourth inequality holds as 
$d=\max\limits_{t \in [T]} d_{t}$ and the last inequality follows from Lemma \ref{L:change}. }

Now, we control the term $\sum\limits_{t=1}^{T}\|\nabla f_{t}(x_{t})\|\delta_{t}$, \cris{which is the fourth term in the right-hand side of \eqref{eq:c8}}.
Applying the Cauchy-Schwarz inequality, the definition of $\Lambda_{T}$, and Assumption \ref{B5} we obtain:
 \begin{align}
   \hspace{-3mm} \left(\sum\limits_{t=1}^{T} \|\nabla f_{t}(x_{t})\|\delta_{t}\right)^{2}\leq \left(\sum\limits_{t=1}^{T} \|\nabla f_{t}(x_{t})\|^{2}\right)\left(\sum\limits_{t=1}^{T} \delta_{t}^{2}\right)  \leq \left(\Gamma\sum\limits_{t=1}^{T} \left( f_{t}(x_{t})-f_{t}(x^{*}_{t})\right)\right)\Lambda_{T}=\Gamma R^{NS}_{T}\Lambda_{T}. \label{eq:b8i}
 \end{align}\noindent Plugging \eqref{eq:b5}, \eqref{eq:b8i}, and \eqref{eq:c8} into \eqref{eq:b2} we get:
   \cris{\begin{align}
  \kappa R^{NS}_{T}  \leq &  \sum\limits_{t=1}^{T}\scal{\nabla f_{t}(x_{t})}{x_{t}-x_{t'}}+\frac{2R^{2}}{\eta}+ \frac{3R}{\eta}V_{T}+\frac{\eta d}{2}\sum\limits_{t=1}^{T}\|\nabla f_{t}(x_{t})\|^{2}+\eta d \left( \Gamma R^{NS}_{T}\Lambda_{T}\right)^{\frac{1}{2}}+\frac{\eta d}{2} \Lambda_{T}\nonumber\\
  &+ 2R \Delta_{T}+\left(2R(d-1)\Gamma\right)^{\frac{1}{2}} \left(V_{T} R^{NS}_{T}\right)^{\frac{1}{2}}\nonumber \\
    \leq &  \sum\limits_{t=1}^{T}\|\nabla f_{t}(x_{t})\|\|x_{t}-x_{t'}\|+\frac{2R^{2}}{\eta}+ \frac{3R}{\eta}V_{T}+\frac{\eta d\Gamma}{2}R^{NS}_{T}+\eta d \left( \Gamma R^{NS}_{T}\Lambda_{T}\right)^{\frac{1}{2}}+\frac{\eta d}{2} \Lambda_{T}+ 2R \Delta_{T}\nonumber\\
  &+\left(2R(d-1)\Gamma\right)^{\frac{1}{2}} \left(V_{T} R^{NS}_{T}\right)^{\frac{1}{2}}
   \label{eq:b8}
 \end{align}}
 where the second inequality follows from Assumption \ref{B5}. To bound the regret, we must control the term $\sum\limits_{t=1}^{T} \|\nabla f_{t}(x_{t})\| \|x_{t}-x_{t'}\|$. Applying the Cauchy-Schwarz inequality and Assumption \ref{B5}, we obtain:
 \begin{align}\label{eq:b13}
   \hspace{-3mm} \left(\sum\limits_{t=1}^{T} \|\nabla f_{t}(x_{t})\|\|x_{t}-x_{t'}\|\right)^{2}\leq&\left(\sum\limits_{t=1}^{T} \|\nabla f_{t}(x_{t})\|^{2}\right)\left(\sum\limits_{t=1}^{T} \|x_{t}-x_{t'}\|^{2}\right) \leq\Gamma R^{NS}_{T}\left(\sum\limits_{t=1}^{T} \|x_{t}-x_{t'}\|^{2}\right). 
 \end{align}
 Then, it remains to bound the second term on the right-hand side of \eqref{eq:b13}:
\begin{align}
\|x_{t'}-x_{t}\|^{2} \leq& (d-1)\sum\limits_{i=t}^{t'-1} \|x_{i+1}-x_{i}\|^{2} \leq (d-1)\sum\limits_{i=t}^{t'-1} \|x_{i+1}'-x_{i}\|^{2} =(d-1)\sum\limits_{i=\max(t,s)}^{t'-1} \|x_{i+1}'-x_{i}\|^{2}\nonumber\\
= & (d-1) \eta^{2} \sum\limits_{i = \max(t,s)}^{t'-1} \left\| \sum\limits_{k  \in \FF_{i}} r_{k} \right\|^{2} \leq  2(d-1) \eta^{2} \sum\limits_{i = \max(t,s)}^{t'-1}  \sum\limits_{k \in \FF_{i}}|\FF_{i}|\left(\|\nabla f_{k}(x_{k}) \|^{2}+\|m_{k} \|^{2} \right)\nonumber \\ \leq & 2d(d-1) \eta^{2} \sum\limits_{i= \max(t,s)}^{t'-1}  \sum\limits_{k \in \FF_{i}} \left(\|\nabla f_{k}(x_{k}) \|^{2}+\delta_{k}^{2}\right),  \label{eq:b10}
\end{align}
  where the first and third inequalities are due to the convexity of the norm. The first equality holds as $\|x_{k+1}-x_{k}\|=0$ for $k \in [s-1]$, the second from \eqref{iterates}, and the final inequality since $|\FF_{k}| \leq d$ for all $k \in [t]$.\medskip
  
\noindent Summing \eqref{eq:b10} from $1$ to $T$ we have that 
\begin{align}
 & \sum\limits_{t=1}^{T}\|x_{t'}-x_{t}\|^{2} \leq 2d(d-1) \eta^{2} 
 \sum\limits_{t=1}^{T} \sum\limits_{i=\max(t,s)}^{t'-1}  \sum\limits_{k \in \FF_{i}}
 \left(\|\nabla f_{k}(x_{k})\|^{2}+\delta_{k}^{2}\right) \nonumber\\ 
 = & 2d(d-1)\eta^{2}\sum\limits_{t=1}^{T}\sum\limits_{i=\max(t,s)}^{t+d_{t} 
 -2}  \sum\limits_{k \in \FF_{i}}\left(\|\nabla f_{k}(x_{k})\|^{2} + \delta_{k}^{2}
 \right) \leq 2d(d-1) \eta^{2} \sum\limits_{t=1}^{T} \sum\limits_{i=\max(t,s)}^{t 
 +d-2} \sum\limits_{k \in \FF_{i}}\left(\|\nabla f_{k}(x_{k})\|^{2}+ \delta_{k}^{2} 
 \right) \nonumber\\
 \leq &  2d(d-1)^{2} \eta^{2}
 \sum\limits_{t=1}^{T} \left\|\nabla f_{t}(x_{t}) \right\|^{2} + 2d(d-1)^{2} 
 \eta^{2} \Lambda_{T}  
 \leq  2 \Gamma d(d-1)^{2}\eta^{2}R^{NS}_{T}+2d(d-1)^{2}\eta^{2}\Lambda_{T} 
 \label{eq:b11}
\end{align}

where the second inequality follows from the definition of $d$, the third by noting that each term is at most $d-1$ times, and the fourth is due to Assumption \ref{B5}. Combining \eqref{eq:b13} and \eqref{eq:b11}, we obtain:
\begin{align}\label{eq:b16}
 \sum\limits_{t=1}^{T} \|\nabla f_{t}(x_{t})\|\|x_{t}-x_{t'}\|\leq& \left( 2d\right)^{\frac{1}{2}}\Gamma(d-1)\eta  R^{NS}_{T}+(2d\Gamma R^{NS}_{T})^{\frac{1}{2}}(d-1)\eta\left(\Lambda_{T}\right)^{\frac{1}{2}}. 
\end{align}

\noindent Substituting  \eqref{eq:b16} into \eqref{eq:b8}, we obtain
  \begin{align}
  R^{NS}_{T}  \leq & \frac{2R^{2}}{\eta\kappa}+\frac{3R}{\eta\kappa }V_{T}+ \frac{2R}{\kappa} \Delta_{T}+ \frac{\eta d}{2\kappa}\Lambda_{T}+ \frac{\eta d\Gamma}{2\kappa}R^{NS}_{T}+\frac{\left( 2d\right)^{\frac{1}{2}}\Gamma(d-1)\eta  }{\kappa}R^{NS}_{T}\nonumber\\
   &+\frac{\left(2R(d-1)\Gamma\right)^{\frac{1}{2}}}{\kappa} \left(V_{T} R^{NS}_{T}\right)^{\frac{1}{2}}+\frac{(2d\Gamma)^{\frac{1}{2}}(d-1)\eta}{\kappa}\left(\Lambda_{T}R^{NS}_{T}\right)^{\frac{1}{2}}+ \frac{\eta d}{\kappa}\left(\Gamma R^{NS}_{T}\Lambda_{T}\right)^{\frac{1}{2}}.\nonumber
 \end{align}
Rearranging terms, we get that $a R^{NS}_{T}-b\left(R^{NS}_{T}\right)^{\frac{1}{2}}\leq c.$  Since $R^{NS}_{T}\geq 0$, by solving the second-order inequality, we get,
  \begin{align}
     R^{NS}_{T}&\leq \left( \frac{b+\sqrt{b^{2}+4ac}}{2a}\right)^{2}\leq \frac{b^{2}+2ac+b\left(b^{2}+4ac\right)^{\frac{1}{2}}}{2a^{2}},\nonumber
 \end{align}
 \end{enumerate} \hspace{9mm} which completes the proof. } \hfill
\end{proof}\medskip

\noindent As an immediate consequence of Theorem \ref{theo:02}, we obtain the following corollary when $m_{k}\equiv 0$.\medskip

\begin{corollary}\label{cor:01}
Under assumptions \ref{B4}-\ref{B3} and considering that the full gradient is available, i.e., for all $t \in [T]$, $r_{t}=\nabla f_{t}(x_{t})$, the following assertions hold: \begin{enumerate} \item If \ref{B2} holds, the stationary regret of Algorithm \ref{DOGD-SC} with a constant step-size $0<\eta_{t}=\eta$ can be upper bounded by:
     \begin{align*}
     R^{NS}_{T} \leq \frac{2R^{2}}{\eta\kappa}+\frac{3RV_{T}}{\eta\kappa}+  \frac{\eta L^{2}Td}{2\kappa}+\frac{L(d-1)V_{T}}{\kappa}+\frac{2\eta(d-1)L^{2}T}{\kappa}\end{align*}
      \item If \ref{B5} holds, the stationary regret of Algorithm \ref{DOGD-SC} with a constant step-size $0<\eta_{t}=\eta< \frac{1}{\bar{\alpha}}=\frac{2\kappa}{\left(d+2d^{\frac{1}{2}}(d-1)\right)\Gamma}$ can be upper bounded by:
\begin{align*}
R^{NS}_{T} \leq \frac{\bar{b}^{2}+2\bar{a}\bar{c}+\bar{b} \left(\bar{b}^{2} +
 4\bar{a}\bar{c}\right)^{\frac{1}{2}}}{2\bar{a}^{2}},
\end{align*}
where $\bar{a} = 1 - \bar{\alpha} \eta > 0$, $\bar{b} = \frac{\left(2R (d-1) 
\Gamma V_{T} \right)^{\frac{1}{2}}}{\kappa},$ and $\bar{c} = \frac{2R^{2} + 3R
 V_{T}}{\kappa}$.
\end{enumerate}
\end{corollary}
\begin{proof}
If $m_{t}=0$, then $\delta_{t}=0$, leading to $\Lambda_{T}=\Delta_{T}=0$. 
Moreover, in \eqref{eq:b10}, it is unnecessary to impose an upper bound of $2$, 
which will impact the computation of $\alpha$ and subsequently affect the value of
 $\bar{a}$.\hfill
\end{proof}

\medskip

\begin{remark}\label{R:4.4} 
In this remark, we only focus on the implications of Corollary \ref{cor:01} under the condition that the error is zero. In the next subsection, we will consider the case where the error is not necessarily zero.
\begin{enumerate}
\item If $f_{t}$ is $L$-Lipschitz for every $t \in [T]$, then the regret bound \eqref{eq:R1} can be interpreted as follows: The term $\frac{2R^{2}}{\eta\kappa}$ represents the initialization error, which decreases as the step-size $\eta$ increases. The second term, $\frac{3RV_{T}}{\eta\kappa}$, incorporates the cumulative path variation $V_{T}$ and also diminishes with a larger $\eta$. In contrast, the terms $\frac{\eta L^{2} Td}{2\kappa}$ and $\frac{2\eta(d-1)L^{2} T}{\kappa}$ grow with $\eta$, suggesting a trade-off between these terms and the first two. The term $\frac{L(d-1)V_{T}}{\kappa}$ highlights that, in the presence of delay ($d > 1$), there is an influence of both the cumulative path variation and the delay, independent of the choice of the step-size, which is different with respect to \cite{pun2024online}.
\item\label{R:4.4(ii)} If $V_{T}$ grows sub-linearly, i.e., $V_{T}=\mathcal{O}(T^{\alpha})$ with $\alpha < 1$, then selecting $\eta=\mathcal{O}(T^{\frac{\alpha-1}{2}})$ ensures a sub-linear regret bound of $\mathcal{O}(dT^{\frac{1+\alpha}{2}})$, as stated in Theorem \ref{theo:02}. Specifically, the optimal regret bound is obtained by setting
$$
\eta=\left(\frac{2R(2R+3V_{T})}{TL^{2}(5d-4)}\right)^{\frac{1}{2}}.
$$
With this choice of $\eta$, Theorem \ref{theo:02} yields a regret bound of 
$$
\frac{\left(2RTL^{2}(5d-4)(2R+3V_{T})\right)^{\frac{1}{2}}}{\kappa}+\frac{L(d-1)V_{T}}{\kappa}.
$$
However, this choice of $\eta$ requires prior knowledge of $V_{T}$, which may not always be available.
\item If $d=1,$ meaning that there is not delay, then \eqref{eq:R2} becomes,
$$R^{NS}_{T} \leq \frac{2R^{2}}{\eta\kappa}+\frac{3RV_{T}}{\eta\kappa}+  \frac{\eta L^{2}Td}{2\kappa},$$
which is the same regret obtained in \cite[Theorem 1 (ii)]{pun2024online}. When $d\geq1$ we extend the results for quasar-convex functions.
\item If $f_{t}$ is $\Gamma$-weakly smooth for every $t \in [T]$, then the regret bound \eqref{eq:R1} can be viewed as the square of the sum of two terms arising from solving a second-order inequality. The first term, $b$, represents the square root of the cumulative path variation and is not null if there exists a delay $(d > 1)$. The second term, $(b^{2}+4ac)^{\frac{1}{2}}$, includes both $b$ and the initialization error along with $V_{T}$. Both terms are dominated by a factor of order $V_{T}^{\frac{1}{2}}$, making the overall regret grow as $\mathcal{O}(dV_{T})$.  Then, if $V_{T}$ grows sub-linearly, the algorithm  ensures a sub-linear regret bound. Since the feasible set is bounded $V_{T}$ is at most linear. Then, similar to \cite[Theorem 1]{pun2024online}, the order of the regret for $\Gamma$-weakly smooth function is better than the one of $L$-Lipschitz. Furthermore, the choice of the step-size $\eta$ is crucial: a smaller $\text{a}=1-\alpha \eta$ can impact the regret bound significantly, as it appears in the denominator, thereby magnifying the overall regret. But, in this setting, the choice of $\eta$ does not require prior knowledge of $V_{T}.$
\item If $d=1$, indicating no delay, then \eqref{eq:R2} simplifies to
$$ R^{NS}_{T} \leq \frac{c}{a} = \frac{2 R^{2} + 3 R V_{T}}{(1 - \alpha \eta) \kappa},$$
which matches the regret bound given in \cite[Theorem 1 (i)]{pun2024online}. Furthermore, the step-size $\eta$ must satisfy the same condition, i.e., $\eta \in \left]0, \frac{2\kappa}{\Gamma}\right[.$
\item Note that for $\kappa=1$ and under the assumption of differentiability, our results apply to star-convex functions with delay. To the best of our knowledge, this is the first work to establish a non-stationary regret bound for this class of functions in the presence of delays. A similar result holds for convex functions; however, in the case where $d=1$, we do not recover the bounds proposed by \cite{besbes2015non}, which are of the order $\mathcal{O}\left(V_{T}^{\frac{2}{3}}T^{\frac{1}{3}}\right)$. In the delayed setting, we extend the work of \cite{wan2023non} by providing non-stationary regret bounds for quasar-convex functions but assuming that $f_{t}$ is $\Gamma$-weakly smooth.
\item  For both regret bounds, a decrease in the constant $\kappa$ indicates that the functions become "more non-convex" resulting in larger regret bounds. This highlights the impact of the level of non-convexity on the performance of the algorithm. 
\end{enumerate}
\end{remark}

\subsection{Extension to the Bandit Setting with Delays}

In this section, we study the zeroth-order version of Algorithm \ref{DOGD-SC}, which allows the method to address more complex scenarios where decision-making relies on partial gradient information and delayed feedback.  To handle the bandit setting, following previous studies \cite{agarwal2010optimal, saha2011improved, wan2022online}, we introduce two changes to Algorithm \ref{DOGD-SC}.\medskip

\noindent First, instead of querying the full gradient $\nabla f_{t}(x_{t})$, we estimate it by querying the function $f_{t}$ at $(p+1)$ points: $x_{t}$ and $\left\{x_{t}+h_{t} e_{i}\right\}_{i=1}^{p},$ where $\{e_{i}\}_{i=1}^{p}$ represents the canonical  basis and $h_{t}>0$ is the discretization parameter. The gradient approximation is given by:
\begin{align}\label{zeroth order}
    r_{t}=\sum\limits_{i=1}^{p} \frac{f_{t}(x_{t}+h_{t} e_{i})-f_{t}(x_{t})}{h_{t}} e_{i},
\end{align}
and arrives at the end of round $t'$.\medskip

\noindent Second, we ensure the feasibility of these points, as $x_{t}$ is feasible, but $\left\{x_{t}+h_{t} e_{i}\right\}_{i=1}^{p}$ may not be. To maintain feasibility, we assume that there exists a $\bar{r}>0$ such that $\mathbb{B}(0,\bar{r})\subset \XX$ and restrict the feasible set to a subset of the original decision set $\XX$, defined as:
\begin{align*}
    \XX_{h}=\left\{\left(1-\frac{h}{\bar{r}}\right)x \mid x \in \XX \right\},
\end{align*}
where $h=\max\limits_{t\in [T]}h_{t} \in ]0, \bar{r}[$. \cris{Since $0\in \XX$ we have that $\XX_{h}\subset \XX$. This ensures that for any $x_{t} \in \XX_{h}\subset \XX$ and for all $i \in [p]$, the point $x_{t}+h_{t}e_{i}$ remains feasible.}\medskip

\begin{theorem}
\label{theo:03}
Let assumptions \ref{B4} and \ref{B3} hold, and assume that each loss function $f_{t} \colon \XX \to \RR$ is $G$-smooth. \cris{Moreover,  assume that there exists a $\bar{r}>0$ such that $\mathbb{B}(0,\bar{r})\subset \XX$.} Then, the stationary regret of Algorithm \ref{DOGD-SC} with a constant step-size $0 < \eta_{t}=\eta < \frac{1}{\alpha_1}=\frac{\kappa}{\left(d+4d^{\frac{1}{2}}(d-1)\right)G}$, projecting over $\XX_{h}$, and $r_{t}$ as in \eqref{zeroth order}, can be upper bounded by:
\begin{align*}
R^{NS}_{T} \leq \frac{b_{1}^{2}+2a_{1}c_{1}+b_{1} \left(b^{2}_{1}+4a_{1}c_{1}\right)^{\frac{1}{2}}}{2a_{1}^{2}},
\end{align*}
where $a_{1}=1-\alpha_{1}\eta>0$, 
$$b_{1}= \frac{\left(2R(d-1)\Gamma V_{T}\right)^{\frac{1}{2}}}{\kappa}+\frac{(2d)^{\frac{1}{2}}(d-1)\eta+\eta d}{\kappa}\left(\Gamma \bar{\Lambda}_{T}\right)^{\frac{1}{2}},\hspace{2mm}\text{ and }\hspace{2mm}c_{1}=\frac{2R^{2}}{\kappa}+\frac{3RV_{T}}{\kappa}+ \frac{\eta d}{2\kappa}\bar{\Lambda}_{T}+ \frac{2R}{\kappa} \bar{\Delta}_{T}.$$
with $\bar{\Lambda}_{T}=\frac{pG^{2}}{4}\sum\limits_{t=1}^{T}h_{t}^{2}$ and $\bar{\Delta}_{T}=\frac{p^{\frac{1}{2}}G}{2}\sum\limits_{t=1}^{T}h_{t}.$
\end{theorem}

\begin{proof}
Since $ \nabla f_{t}$ is $G$-Lipschitz, it follows from \cite[Subsection 1.1.2]{polyak1987introduction}, that for every $x \in \RR^{p}$,  $t\in [T]$, and $i\in [p]$
\begin{align*}
\left| f_{t}(x_{t}+h_{t}e_{i})-f_{t}(x_{t})-h_{t} \scal{\nabla f_{t}(x_{t})}{e_{i}}  \right| \leq \frac{G h_{t}^{2}}{2},
\end{align*}
which implies that
\begin{align*}
 \left\|\sum\limits_{i=1}^{p} \left[\frac{f_{t}(x_{t}+h_{t} e_{i})-f_{t}(x_{t})}{h_{t}}
 e_{i}\right]-\nabla f_{t}(x_{t})\right\|=&\sqrt{\sum\limits_{i=1}^{p} \left( \frac{f_{t} 
 (x_{t} + h_{t}e_{i})-f_{t}(x_{t})}{h_{t}}-\scal{\nabla f_{t}(x_{t})}{e_{i}} 
 \right)^{2}} \leq \frac{p^{\frac{1}{2}} G h_{t}}{2}.
\end{align*}
\cris{Moreover, since $0\in \XX$ we have that $x_{t} \in\XX_{h}\subset \XX$. Thus, for any $t\in [T]$ and $i \in [p]$, the point $x_{t}+h_{t}e_{i}\in\XX$.} Then, we can set $\delta_{t}\equiv \frac{p^{\frac{1}{2}}Gh_{t}}{2}$, which gives that $\bar{\Lambda}_{T}=\frac{pG^{2}}{4}\sum\limits_{t=1}^{T}h_{t}^{2}$ and $\bar{\Delta}_{T}=\frac{p^{\frac{1}{2}}G}{2}\sum\limits_{t=1}^{T}h_{t}.$ On the other hand, from \cite[Proposition 1]{pun2024online} $G$-smooth implies $2G$-weakly smooth. We conclude the proof applying Theorem \ref{theo:02} \ref{T:Weakly-smooth}.\hfill
\end{proof}

\medskip

\begin{remark} 
In this remark, we explore the implications of our formulation.
\begin{enumerate}
    \item \cris{It follows from \cite[Proposition~1]{pun2024online} that $G$-smoothness implies $2G$-weak smoothness. Hence, Theorem~\ref{theo:02}\ref{T:Weakly-smooth} can be applied under assumption, that every $f_{t}$ is $G$-smooth in place of assumption~\ref{B5}.}

    \item By  Theorem \ref{theo:03}, for every $x \in \RR^p$ and sufficiently small $h_{t}> 0$, we have:
    $$\sum\limits_{i=1}^{p}\left[ \frac{f_{t}(x_{t}+h_{t} e_{i})-f_{t}(x_{t})}{h_{t}} e_{i} \right]-\nabla f(x)=\mathcal{O}(h_{t}^{2}).$$
    Therefore, when $ h_ {t}> 0 $, the gradient approximation is biased, so our method is more flexible than the stochastic methods that require unbiased estimators, such as \cite{pun2024online}.
    \item Theorem \ref{theo:02} also applies to a $2p$-point gradient estimator:
    $$\sum\limits_{i=1}^{p} \left[ \frac{f(x_{t}+h_{t} e_{i})-f_{t}(x_{t}-h_{t} e_{i})}{2h_{t}} e_{i} \right].$$
From \cite[Subsection 1.1.2]{polyak1987introduction}, for every $x \in \RR^p$ and $i \in [p]$:
\begin{align*}
 & \hspace{1.0cm} \left| f(x+h e_{i})-f(x-h e_{i})-2h \scal{\nabla f(x)}{e_{i}} \right| 
 \leq G h^{2} \\ 
 & \Longrightarrow \, \left\| \sum\limits_{i=1}^{p} \left[ \frac{f(x+h e_{i})-f_{t} (x_{t}-he_{i})}{2h} e_{i} \right]-\nabla f(x) \right\| \leq \frac{p^{\frac{1}{2}} G h}{2},
\end{align*}
and we proceed analogously to Theorem \ref{theo:03}. However, the main drawback 
is the increased number of function queries required.
\item From the definitions of $a_1$, $b_1$, and $c_1$, we deduce that $a_1$ is constant, $b_1$ scales with the maximum of $\bar{\Lambda}_{T}$ and $V_{T}^{\frac{1}{2}}$, and $c_1$ scales with the maximum of $\bar{\Delta}_{t}$, $\bar{\Lambda}_{T}$, and $V_{T}$. Consequently, the overall regret grows as $\mathcal{O}(\max\{V_{T}, \bar{\Delta}_{T},\bar{\Lambda}_{T}\})$, making the choice of $\left\{h_{t}\right\}_{t\in [T]}$ crucial. If $V_{T}$ grows sub-linearly, i.e., $V_{T}=\mathcal{O}(T^\alpha)$ with $\alpha < 1$, selecting $h_{t} \equiv \mathcal{O}(T^{\alpha-1})$ ensures a sub-linear regret of $\mathcal{O}(T^\alpha)=\mathcal{O}(V_{T})$. However, this choice of $h_{t}$ requires prior knowledge of $V_{T}$, which may not always be available.
\item \cris{We extend the results of \cite{wan2023non} to quasar-convex functions in the bandit setting. In addition, we complement the results of \cite[Theorem~2]{wan2022online}, which address the bandit setting with delays and stationary regret for strongly convex,  to quasar-convex functions in the bandit setting with delays and dynamic regret. Furthermore, it is important to note that the sum of quasar-convex functions may not be quasar-convex, which means that we could not directly apply the techniques of \cite[Theorem 2]{wan2022online}.}
\end{enumerate}
\end{remark}

\medskip

\noindent Now, we study the convergence of the delayed zeroth-order method for strongly quasar-convex functions in an offline setting, which will be applied to determine the minimizer for applications in quadratic fractional functions. 
\medskip

\begin{corollary}\label{cor1}
    Let $f \colon \XX \to \RR$ be $(\kappa, \mu)$-strongly quasar-convex with respect to $x^{*} \in \argmin\limits_{x \in \XX} f(x)$, and assume that $f$ is $G$-smooth.  Let $\left\{r_{t}\right\}_{t \in \NN}$ be defined as in \eqref{zeroth order} by the sequence $\left\{h_{t}\right\}_{t \in \NN}$, where $h_t$ is upper bounded by a constant $h$. Let the convex decision set $\XX$ be bounded with radius $R$. Consider the iterates $\left\{x_{t}\right\}_{t \in \NN}$ generated by Algorithm \ref{DOGD-SC} using a constant step-size $0 < \eta_{t} = \eta < \frac{1}{\alpha_1} = \frac{\kappa}{\left(d + 4 d^{\frac{1}{2}}(d-1)\right)G}$, where $f_{t} \equiv f$, with projection onto $\XX_{h}$. Assume further that $\bar{\Delta}_{t}$ and $\bar{\Lambda}_{t}$ are sub-linear terms.
Then, $X_{t} \to x^{*}$ as $t \to +\infty$, where $X_{T} = \frac{1}{T} \sum\limits_{k=1}^{T} x_{t}$ for all $T > 0$.
\end{corollary}

\medskip

\begin{proof}
 By the convexity of the norm, the $(\kappa, \mu)$-strong quasar-convexity of $f$ with respect to $x^{*}$, and \cite[Lemma 1 (ii)]{pun2024online}, for any $T > 0$, we obtain:
\begin{align}\label{R: convergence of iterates}
\|X_T - x^{*}\|^2 \leq \frac{\sum\limits_{t=1}^{T}\|x_t - x^{*}\|^2}{T} \leq \frac{8G \sum\limits_{t=1}^{T}(f(x_t) - f(x^{*}))}{T\kappa^2 \mu^2} \leq \frac{8G R_T}{T\kappa^2 \mu^2}.
\end{align}
Since $f_t \equiv f$ we obtain that $V_T = 0$. Moreover, since  $\bar{\Lambda}_T$ and $\bar{\Delta}_T$ are sub-linear terms, the the right-hand side of \eqref{R: convergence of iterates} tends to zero, thereby ensuring the convergence of the average iterate.
 \hfill
\end{proof}

\medskip

\begin{remark}
   As far as we know, this is the first theorem that proved the convergence of the iterates of a delayed zeroth-order version of gradient descent for quasar-convex functions. However, the proof relies on making the regret tends to zero over time, which does not fully exploit the advantages of $(\kappa,\mu)$-strong quasar-convexity. For example, in \cite[Proposition 2]{pun2024online}, it is shown that under similar assumptions, an exponential convergence rate can be achieved when the full gradient is available and there is no delay.
\end{remark}

\section{Examples and Applications}\label{sec:ea}

\subsection{New Examples of Quasar-Convex Functions}

In this section, we prove that strongly quasiconvex functions are (strongly) 
quasar-convex on convex and compact sets. To do this, we first present the 
following result, its easy proof is omitted. 
\medskip
\begin{proposition}\label{prop:lc}
 Let $f$ be a differentiable function such that there exists $\gamma, \kappa > 0$ 
 such that
\begin{align}
 & f(x) \leq f(y) ~ \Longrightarrow ~ \scal{\nabla f(y)}{x-y}
  \leq -\frac{\gamma}{2}\|x-y\|^{2}, \label{gen:char} \\
 & f(x) \leq f(y) ~ \Longrightarrow ~ \scal{\nabla f(y)}{x-y}\leq\kappa (f(x)-f(y)). \label{kappa:qc}
\end{align}
Then
\begin{align*}\label{s:qcx}
 f(x) \leq f(y) ~ \Longrightarrow ~  f(x) \geq f(y)+\frac{2}{\kappa} \scal{\nabla f(y)}{x-y}+ \frac{\gamma}{2 \kappa}\|x-y\|^{2}.
\end{align*}
\end{proposition}


\medskip
\noindent In the next theorem, we prove that differentiable strongly quasiconvex functions are quasar-convex:

\medskip

\begin{theorem}\label{sqcx:qconvexT}
 Let $\XX \subseteq \RR^{p}$ be a convex and compact set, $f\colon \RR^{p} 
 \rightarrow \RR$ be a differentiable strongly quasiconvex with modulus 
 $\gamma > 0$ and $\argmin\limits_{\XX}\,f=\{x^{*}\}$. Then there exists $n_{1} 
 \in \mathbb{N}$ such that 
 \begin{align*}
  \scal{\nabla f(y)}{x^{*}-y} \leq \frac{2}{n} \left( 
  f(x^{*})-f(y) \right), ~ \forall ~ n \geq n_{1}.
 \end{align*}
 As a consequence, $f$ is quasar-convex with modulus $\kappa=\frac{2}{n_{1}+2} 
 \in \, ]0, 1[$ and $\gamma > 0$ on $\XX$.
\end{theorem}

\begin{proof}
By Lemma \ref{exist:unique}, $\argmin\limits_{\XX}\,f=\{x^{*}\}$. Then, by equation \eqref{qgc}, we have
 \begin{align*}
  & f(x^{*})+\frac{\gamma}{4}\|x^{*}-y\|^{2} \leq f(y) ~ \Longleftrightarrow ~ 0 < \frac{\gamma}{2}\|x^{*}-y
  \|^{2} \leq 2(f(y)-f(x^{*})), ~ \forall ~ y \in \XX.
 \end{align*}
 Since $\XX$ is compact and $f$ is continuous, there exists $M>0$ such that $f(y) \leq M$ 
 for all $y \in \XX$, thus,
 \begin{align*}
  & 0 < \frac{\gamma}{2}\|x^{*}-y\|^{2} \leq 2(f(y)-f(x^{*})) \leq 2(M-f(x^{*})), ~ \forall ~ y \in \XX.
 \end{align*}
 
\noindent By Archimedean axiom, there exists $n_{1} \in \mathbb{N}$ such that,
 $$0 < \frac{2}{n} (M-f(x^{*})) \leq \frac{\gamma}{2}\|x^{*}-y\|^{2}, ~ \forall ~ n \geq n_{1}.$$
 Using now relation \eqref{gen:char} (which holds in view of Lemma \ref{char:gradient}), we have
 \begin{align*}
  0 < & ~  \frac{2}{n} (M-f(x^{*})) \leq \frac{\gamma}{2}\|x^{*}-y\|^{2} \leq \scal{\nabla f(y)}{y-x^{*}}, ~ \forall ~ n \geq n_{1}. 
 \end{align*}
  Finally, since $f(y) \leq M$ for all $y \in \XX$, we have for all $n \geq n_{1}$ that
  $$\scal{\nabla f(y)}{x^{*}-y} \leq \frac{2}{n} 
  (f(x^{*})-M) \leq \frac{2}{n} (f(x^{*})-f(y)), ~ \forall ~ y \in 
  \XX.$$
\noindent \cris{Now, since $\kappa$ must lie in the interval $]0,1]$, we choose $\kappa=\frac{2}{n_{1}+2} \in \, ]0, 1[$. Then, $f$ is $\kappa$-quasar convex on $\XX$.} \hfill
\end{proof}
\medskip 

\noindent In the unbounded case, we have the following sufficient condition:

\medskip

\begin{proposition}\label{non:decerasing}
 Let $f\colon \RR^{p} \rightarrow \RR$  be a differentiable and strongly quasiconvex with modulus $\gamma > 0$. \cris{If $\lambda \mapsto f(x+\lambda u)$, for $\lambda>0$, is non\-de\-crea\-sing for any $(x, u) \in \RR^{p}\times\RR^{p}$, then $f$ sa\-tis\-fies relation \eqref{kappa:qc}.}
\end{proposition}

\begin{proof}
 Let $(x, y) \in \XX^{2}$ be such that $x \in S_{f(y)} (f)$. Then for all $\lambda \in [0, 1]$, we have ($v=y-x$)
 \begin{align*}
  f(x) \leq f( \lambda y + (1-\lambda)x) & \leq \max\{f(y), f(x)\} - \lambda 
  (1-\lambda) \frac{\gamma}{2} \| y - x \|^{2} \\
  & \leq f(y) - \lambda (1-\lambda) \frac{\gamma}{2} \| y - x \|^{2}.
 \end{align*}
 By taking $\lambda =\frac{1}{2}$, we have
 \begin{align*}
  0 < \frac{\gamma}{2} \| y - x \|^{2} \leq 4(f(y) - f(x)), ~ \forall 
  ~ x \in S_{f(y)} (f).
 \end{align*}
 As in the previous proof, and taking into consideration relation \eqref{gen:char}, there exists $n_{0} \in \mathbb{N}$ such that
 $$0 < \frac{4}{n} (f(y) - f(x)) \leq \frac{\gamma}{2} \| y - x \|^{2}
  \leq \scal{\nabla f(y)}{y - x}, ~ \forall ~ x \in S_{f(y)} (f), ~ \forall ~ 
 n \geq n_{0},$$
 i.e., there exists $\kappa > 0$ such that $\kappa (f(y) - f(x)) \leq \scal{\nabla f(y)}{y - x}$ for all $x \in S_{f(y)} (f)$, i.e., $f$ satisfies relation \eqref{kappa:qc}. \hfill
\end{proof}

\noindent As a consequence, we have the desired result.
\medskip

\begin{corollary}\label{sqcx:qconvexC}
 Let $\XX \subseteq \RR^{p}$ be a convex set, $f\colon \RR^{p} \rightarrow \RR$ be a differentiable strongly quasiconvex with modulus 
 $\gamma > 0$ and $\argmin\limits_{\XX}\,f=\{x^{*}\}$. If at least one of the following assertions hold:
 \begin{itemize}
  \item[$(a)$] $\XX$ is bounded.
  
 \item[$(b)$] $\XX = \RR^{p}$ and \cris{$\lambda \mapsto f(x+\lambda u)$, for $\lambda>0$, is non\-de\-crea\-sing for any $(x, u) \in \RR^{p}\times\RR^{p}$.}
 \end{itemize}
 Then there exists $n_{1} \in \mathbb{N}$ such that $f$ is strongly quasar-convex with modulus
 $\kappa=\frac{1}{n_{1}+2} \in \, ]0, \frac{1}{2}[$ and $\frac{\gamma}{\kappa} > 0$ on $\XX$.
\end{corollary}

\begin{proof}
 $(a)$ (resp. $(b)$): Since $f$ is a differentiable strongly quasiconvex function with modulus $\gamma > 0$, it satisfies relation \eqref{gen:char} by Lemma \ref{char:gradient}, while in view of Theorem \ref{sqcx:qconvexT} (resp. Proposition \ref{non:decerasing}), there exists $n_{1} \in \mathbb{N}$ for which $f$ satisfies relation \eqref{kappa:qc} when $x=x^{*}$ with $\kappa=\frac{2}{n_{1}+2} \in \, ]0, 1[$. Therefore, it follows from Proposition \ref{prop:lc} (applied just for $x=x^{*}$) that 
 \begin{equation*}
  f(x^{*}) \geq f(y)+\frac{2}{\kappa} \scal{\nabla f(y)}{x^{*}-y}+\frac{\gamma}{2 \kappa}\|x^{*}-y\|^{2}, ~ \forall ~ y \in \XX,
\end{equation*}
i.e., $f$ is strongly quasar-convex with modulus $\frac{\gamma}{\kappa} > 0$ and $\kappa=\frac{1}{n_{1}+2} \in \, ]0, \frac{1}{2}[$ on $\XX$. \hfill
\end{proof}

\medskip

\begin{remark}
 \begin{enumerate}
  \item The reverse statement in the previous results does not hold in general. Indeed, the function in Example \ref{ex1} is strongly quasar-convex and it is not strongly quasiconvex (it is not even quasiconvex) since there are sublevel sets that are not convex, while the function $f(x) = (x^{2}+\frac{1}{8})^{1/6}$ is quasar-convex on $\mathbb{R}$ with $\kappa = 1/2$, and it is not strongly quasiconvex as a consequence of \cite[Theorem 1]{lara2022strongly}.

  \item Another sufficient condition for differentiable strongly quasiconvex functions to be quasar-convex on the whole space $\mathbb{R}^{p}$ is that their gradient should be Lipschitz continuous (see \cite[page 15]{lara2024characterizations}). The extension of Corollary \ref{sqcx:qconvexC} and/or Theorem \ref{sqcx:qconvexT} to the whole space remains as open problems.
 \end{enumerate}
\end{remark}

\medskip

\noindent  As noted in Corollary \ref{sqcx:qconvexC}, even when differentiable strongly quasiconvex functions are strongly quasar-convex, it could be hard to determine the parameters $\kappa \in \, ]0, 1]$ and $\gamma > 0$ of strong quasar-convexity. However, in the next proposition, we show how to compute these parameters for the case of quadratic fractional functions. \medskip

\begin{proposition}\label{prop:param}
 Let $A, B \in \RR^{p\times p}$, $a, b \in 
  \RR^{p}$, $\alpha, \beta \in \RR$, and $f\colon 
  \RR^{p} \rightarrow \RR$ be the functions given by:
  \begin{align*}
   f(x)=\frac{\ft(x)}{\fh(x)}=\frac{\frac{1}{2} \scal{Ax}{x}+\scal{a}{x}+\alpha}{\frac{1}{2} \scal{Bx}{x}+\scal{b}{x}+\beta}.
  \end{align*}
  Take $0 < m < M$ and define:
  $$K := \{x \in \RR^{p}\mid ~ m \leq \fh(x) \leq M\}.$$ 
   If $A$ is a positive definite matrix and at least one of the following 
   conditions holds:
  \begin{enumerate}
   \item[$(a)$] $B=0$ (the null matrix),
 
   \item[$(b)$] $\ft$ is nonnegative on $K$ and $B$ is negative semi-definite,
 
   \item[$(c)$] $\ft$ is non-positive on $K$ and $B$ is positive semi-definite,
 \end{enumerate}
  then the function $f$ is $\kappa$-quasar-convex on $K$ with $\kappa = \frac{m}{M}$ and is $(\frac{\kappa}{2}, \gamma)$-strongly quasar-convex on $K$, where $\gamma = \frac{\sigma_{\min}(A)}{m}$. Moreover, for every $\lambda \in ]0,1[$, $f$ is $(\lambda \kappa, (\lambda^{-1} - 1)\frac{\sigma_{\min}(A)}{4M})$-strongly quasar-convex on $K$.
\end{proposition}

\begin{proof}
Since $\fh(x) > 0$ for all $x \in \dom(\ft) \cap \XX$, we have $\dom(f) = \dom(\ft) \cap \XX=\XX$. By Remark \ref{rem:exam} \ref{3.3 (ii)}, $f$ is strongly quasiconvex on $\XX$, ensuring a unique minimizer $x^{*} \in \argmin\limits_\XX f(x)$. For every $\lambda \in [0, 1]$, it holds that  
\begin{align*}
 & \hspace{2.0cm} f(\lambda x^{*} + (1-\lambda) x) \leq \frac{\lambda \ft(x^{*}) + (1-\lambda) \ft(x)}{\lambda \fh(x^{*}) + (1-\lambda) \fh(x)} \\ 
 & \Longleftrightarrow \, f(\lambda x^{*} + (1-\lambda) x) \leq f(x) + \frac{\lambda \fh(x^{*})}{\lambda \fh(x^{*}) + (1-\lambda) \fh(x)} \left(f(x^{*}) - f(x)\right).
 \end{align*}
Using that $\fh(x^{*}) \geq m$ and $\fh(x) \leq M$, it follows that  
$$f(\lambda x^{*} + (1-\lambda) x) \leq f(x) + \lambda \frac{m}{M} \left(f(x^{*}) - f(x)\right).$$
By Lemma \ref{L2}, $f$ is $\kappa$-quasar-convex on $\XX$ with $\kappa = \frac{m}{M}$. Moreover, applying lemmas \ref{char:gradient} and \ref{L2}, we have 
$$\scal{\nabla f(x)}{x^{*} - x} \leq -\frac{\sigma_{\min}(A)}{2M} \|x - x^{*}\|^2.$$
Combining the above, we get
$$\frac{\sigma_{\min}(A)}{2M\kappa} \|x - x^{*}\|^2 + \frac{2}{\kappa} 
\scal{\nabla f(x)}{x^{*} - x} \leq \lambda \kappa \left(f(x^{*}) - f(x)\right)$$
which implies that $f$ is $(\frac{\kappa}{2}, \gamma)$-strongly quasar-convex. 
\medskip

\noindent Finally, for fixed $\lambda \in ]0,1[$, we have  
$$\scal{\nabla f(x)}{x - x^{*}} \geq \kappa \left(f(x) - f(x^{*})\right) = \kappa \lambda \left(f(x) - f(x^{*})\right) + \left(1-\lambda\right)\kappa \left(f(x) - f(x^{*})\right),$$
where the last term satisfies  
$$\left(1-\lambda\right)\kappa \left(f(x) - f(x^{*})\right) \geq \frac{\left(1-\lambda\right)\kappa \sigma_{\min}(A)}{8M} \|x - x^{*}\|^2.$$
Thus, $f$ is $\left(\lambda\kappa, \frac{\left(1-\lambda\right)\kappa\sigma_{\min}(A)}{4M\lambda}\right)$-strongly quasar-convex on $\XX$.  \hfill
\end{proof}

\medskip 

\begin{remark}
The parameter $\lambda$ introduces a trade-off between $\kappa$ and $\gamma$. By setting $\lambda = 1$, we obtain the largest constant for quasar-convexity, but losing strong quasar-convexity. In contrast, setting $\lambda \leq \frac{m}{4M+m}$, we achieve a larger constant for strong convexity, greater than $\gamma$, but this significantly reduces the constant for quasar-convexity.  By selecting the constants as $(\frac{\kappa}{2}, \gamma),$ the quasar-convexity constant becomes larger than the previous choice. However, the strong quasar-convexity parameter cannot be arbitrarily chosen or adjusted in this case.
\end{remark}

\medskip

\noindent In the next example, we check that quadratic fractional functions satisfies the assumptions of our algorithm.

\medskip 

\begin{example}
 Let $A, B \in \RR^{p\times p}$, $a, b \in 
  \RR^{p}$, $\alpha, \beta \in \RR$, $f\colon 
  \RR^{p} \rightarrow \RR$, $0 < m < M$ and $K$ defined as in Proposition \ref{prop:param}. Then the gradient is given by:
$$\nabla f(x) = \frac{\fh(x)(Ax + a) - \ft(x)(Bx+b)}{(\fh(x))^2}.$$ 
Additionally, if $K$ is bounded by $R$, then $f$ is $L$-smooth, where
\begin{align*} 
 L =&\frac{\|A\|^2}{m} + 2 \frac{(\|B\|R+\|b\| )(\|A\| R+ \|a\|)}{m^2} + 
 \frac{\left( \|A\| R^2 + 2\|a\| R + 2\alpha \right) \left(\|B\|R+\|b\|\right)^2}{m^3}
 + \frac{\left( \frac{1}{2}\|A\| R^2 + \|a\| R + \alpha\right) \|B\|^{2}}{m^2}.\end{align*}
Indeed, the Hessian $f(x) $ is given by:
\begin{align*}
 \nabla^2 f(x) =& \frac{A^{\top}A}{\fh(x)} - \frac{\left(Bx+b\right)(Ax 
 + a)^\top}{(\fh(x))^2} - \frac{\left(Ax + a\right) \left(Bx+b\right)^\top}{(\fh(x))^2} 
 - \ft(x) \frac{B^{\top}B}{(\fh(x))^2}+ 2\ft(x) \frac{\left(Bx+b\right) \left(Bx +
 b\right)^\top}{(\fh(x))^3}. 
\end{align*}
We now proceed to bound the Hessian, assuming that $K$ is bounded. Since $\fh(x) \geq m$, we can bound the terms as follows:
\begin{align*}\|\nabla^2 f(x)\| \leq &\ 2 \frac{\left(\|B\|\|x\|+\|b\|\right)\left(\|A\| \|x\| + \|a\|\right)}{m^2} + 2\frac{\left( \frac{1}{2} \|A\| \|x\|^2 + \|a\| \|x\| + \alpha \right)\left(\|B\|\|x\|+\|b\|\right)^2}{m^3}\nonumber\\&\frac{\|A\|^2}{m} ++\frac{\left( \frac{1}{2} \|A\| \|x\|^2 + \|a\| \|x\| + \alpha \right)\|B\|^2}{m^2}\leq L.\end{align*}
Thus, $f$ is $L$-smooth. 
\end{example}

\subsection{Other Non-Convex Examples}

The following two non-convex examples of quasar-convex functions are already known, 
we will check that they satisfies the assumptions of our algorithm.

\medskip

\begin{example}\label{Example 5} We set $\XX = \{x \in \RR^{p} \mid \|x\| \leq R\}$ and set the function $f \colon \XX \rightarrow \RR$. The function is defined as: \begin{align}\label{ex:function1} f(x) = \ft(\|x\|) \fh \left( \frac{x}{\|x\|} \right), \end{align} where $\ft(t) = \frac{t^{2}}{1+t^{2}}$ and $\fh(x) = \sum\limits_{i=1}^{p} a_{i} \sin^{2}(b_{i} x_{i})$ with $\{a_{i}\}_{i=1}^{p} \subseteq [0, m_{1}]^{p}$ and $\{b_{i}\}_{i=1}^{p} \subseteq [-m_{2}, m_{2}]^{p}$. \cris{Since $x^{*} = 0$ and $f(x^{*})=0$, quasar-convexity by definition is equivalent to finding $\kappa$ such that, for every $x \in [-R,R]$,}
\cris{\[
\kappa \leq \frac{\scal{x}{\nabla f(x)}}{f(x)} = \frac{2}{x^{2}+1}.
\]}

\noindent \cris{This implies that $f_{t}$ is $\tfrac{2}{R^{2}+1}$-quasar-convex, and consequently $f$ is also $\tfrac{2}{R^{2}+1}$-quasar-convex on $\XX$ with respect to $x^{*}=0$ (see Proposition~\ref{prop1}), with $V_{T}=0$.} The gradient of $f$ is given by:
$$\nabla f(x) = \frac{x}{\|x\|} \ft'(\|x\|) \fh\left(\frac{x}{\|x\|}\right) + \frac{1}{\|x\|} \left(\Id - \frac{xx^\top}{\|x\|^{2}}\right) \ft(\|x\|) \nabla \fh\left(\frac{x}{\|x\|}\right)$$
where, for every $u \in \RR^{p}$ and $\alpha>0$, $\nabla \fh(u)$ and $\ft'(\alpha)$  are given by: \begin{align*} \nabla \fh(u) &= \left( a_{i} b_{i} \sin(2b_{i} u_{i}) \right)_{i=1}^{p} \text{ and }\ft'(\alpha) = \frac{2\alpha}{(\alpha^{2}+1)^{2}}.\end{align*}

\noindent Note that, by defining  $\nabla f(0)=0$ the gradient is continuous on $\XX$. Moreover, the function $f$ is $(m_{1}p+m_{1}m_{2}p^{\frac{1}{2}})$-Lipschitz continuous. Indeed,
 \begin{align*} \|\nabla f(x)\| &\leq \left\|\frac{x}{\|x\|}\right\| \left\|\ft'(\|x\|)\right\| \left\|\fh\left(\frac{x}{\|x\|}\right)\right\| + \left\|\frac{\ft(\|x\|)}{\|x\|}\right\| \left\|\Id - \frac{xx^\top}{\|x\|^{2}}\right\| \left\|\nabla \fh\left(\frac{x}{\|x\|}\right)\right\| \nonumber \\ &\leq \left\| \fh\left(\frac{x}{\|x\|}\right)\right\| + \left\|\nabla \fh\left(\frac{x}{\|x\|}\right)\right\|\leq m_1 p +m_1 m_2 p^{\frac{1}{2}}. \end{align*}
\end{example}
\medskip

\begin{example} 
In this example, we focus on GLM using the logistic function, which is an example of 
a quasar-convex and weakly-smooth function \cite{wang2023continuized}, and provides an instance 
where we can control the path variation $V_{T}$.  We set the feasible set as 
$\XX=\{x \in \RR^{p} \mid \|x\| \leq R\} $. The function $ f \colon \XX \to \RR $ 
is defined as follows:
\begin{align} 
f(x) = \frac{1}{2m} \sum\limits_{i=1}^{m} \left( \sigma(\scal{a_{i}}{x}) - 
b_{i} \right)^{2}, \label{ex:f2}
\end{align}
where, for every $i\in [m]$, $a^{i}\in\RR^{p}$ and $b^{i} = \sigma(
\scal{a_{t}^{i}}{x^{*}})$, with $\sigma$ denoting the logistic function defined 
as follows $\sigma(\alpha)=\frac{e^{\alpha}}{1+e^{\alpha}}$. Note that 
$\sigma'(\alpha)=\frac{e^{\alpha}}{(1+e^{\alpha})^{2}}$ and $\sigma''(\alpha)= \frac{e^{\alpha}(1-e^{\alpha})}{(1+e^{\alpha})^{3}}$, which implies that 
$\sigma'$ is bounded by $\frac{1}{4}$. The gradient of $f$ is given by:
\begin{align*} 
\nabla f(x) = & \frac{1}{m} \sum\limits_{i=1}^{m} \sigma' (\scal{a_{i}}{x})\left( \sigma(\scal{a_{i}}{x})-b_{i} \right)a_{i}^{\top},
\end{align*}
Hence $f$ is $\Gamma$-weakly smooth with $\Gamma = \frac{\max\limits_{i\in[m]} 
\|a_{i}\|^{2}}{8}$. Indeed, since $4\sigma^{\prime} \leq 1,$ and $f(x^{*})=0$, 
we have, for every $x\in \XX$, that
\begin{align*}
\|\nabla f(x)\|^{2} \leq & \frac{1}{m} \sum\limits_{i=1}^{m} \left\|\sigma' (\scal{a_{i}}{x}) \left( \sigma (\scal{a_{i}}{x}) - b_{i} \right) a_{i}^{\top} \right\|^{2} 
\leq \frac{1}{m} \sum\limits_{i=1}^{m} \sigma'(\scal{a_{i}}{x})^{2} \left( \sigma(\scal{a_{i}}{x}) - b_{i} \right))^{2} \left\|a_{i}^{\top} \right\|^{2} 
\nonumber\\ 
\leq &  \frac{\max\limits_{i\in[m]}\|a_{i}\|^{2}}{16m} \sum\limits_{i=1}^{m}\left( \sigma(\scal{a_{i}}{x})-b_{i} \right))^{2} 
= \frac{\max\limits_{i\in[m]}\|a_{i}\|^{2}}{8} \left(f(x) -f(x^{*}) \right),
\end{align*}
Since $f(x^{*})=0$ and the positiveness of $f$, for $\kappa=\min(1,8\sigma'(R)),$ we have that 
\begin{align*}
 f(x^{*}) \geq & \left(1 - \frac{8 \sigma'(R)}{\kappa} \right) f(x) \geq - \frac{2
 \sigma'(R)}{\kappa} \frac{1}{2m} \sum\limits_{i=1}^{m} \left( \sigma(\scal{a_{i}}{x}) 
 - \sigma (\scal{a_{i}}{x^{*}}) \right) \left(\scal{a_{i}}{x} - \scal{a_{i}}{x^{*}} 
 \right) + f(x) \nonumber\\ 
 \geq & -\frac{1}{\kappa}\frac{1}{2m} \sum\limits_{i=1}^{m} 2\sigma'(\scal{a_{i}}{x})) 
 \left( \sigma(\scal{a_{i}}{x})- \sigma(\scal{a_{i}}{x^{*}}) \right)\left( 
 \scal{a_{i}}{x}- \scal{a_{i}}{x^{*}} \right) +f(x)\nonumber\\ 
 \geq & \frac{1}{\kappa}\scal{\nabla f(x)}{x^{*}-x} +f(x),
\end{align*}
which implies that $f$  is $\kappa$-quasar-convex. 
\end{example}

\section{Numerical Experiments}\label{sec:ne}

 In this section, we present numerical experiments to evaluate the performance of Delayed Online Gradient Descent on quasar-convex functions. All numerical examples are implemented in Python on a laptop equipped with an AMD Ryzen 5 3550 Hz processor, Radeon Vega Mobile Gfx, and 32 GB of RAM. The corresponding code can be downloaded from this \href{https://github.com/cristianvega1995/Delayed-Feedback-in-Online-Quasar-Convex-Optimization-A-Non-Stationary-Approach}{repository}. \medskip

\noindent We conduct experiments in three different settings: highly non-convex functions, GLM, and quadratic fractional functions. In all experiments, unless otherwise specified, we set $T = 20,000$ and run the algorithm with $d \in \{1, 5, 10, 20\}$. At each iteration, we set $r_{t} = \nabla f_{t}(x_{t})$, and the delay $d_t$ is drawn from a discrete uniform distribution over $[d]$. Each experiment is repeated $20$ times with random realizations of the sequence $\{f_{t}\}_{t \in [T]}$, and we report both the average performance and its standard deviation. \medskip

\noindent For the first experiment, we consider the Lipschitz continuous function presented in \eqref{ex:function1}. We set $R = 100$ and $p = 100$. \cris{For every $t \in [T]$, we define $f_{t}$ using parameters $\{a^{t}_{i}\}_{i=1}^{p}$ and $\{b^{t}_{i}\}_{i=1}^{p}$, which are drawn independently and uniformly from $[0,1]^{p}$ and $[-2.5,2.5]^{p}$, respectively.} Since $x^{*}_{t} = 0$, we have $f_{t}(x^{*}_{t}) = 0$ for every $t \in [T]$, and thus $V_{T} = 0$. The initialization is given by $x_{0} \sim U(0.2,0.4)^{p}$. The step-size is chosen as 
\[
\eta = \frac{2R}{L}\,(T(5d-4))^{-\tfrac{1}{2}},
\]
which, according to Remark \ref{R:4.4} \ref{R:4.4(ii)}, is the optimal step-size for minimizing the upper bound of the regret.

 \begin{figure}[htbp]
    \centering
    \includegraphics[scale=0.5]{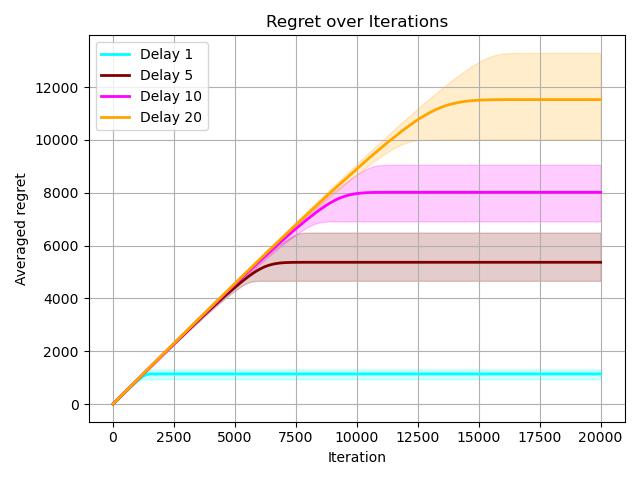}
    \caption{Average regret of Algorithm \ref{DOGD-SC} on the Lipschitz continuous function defined in \eqref{ex:function1} withh $d \in \{1, 5, 10, 20\}$.}
    \label{fig:noteo}
\end{figure}
\begin{table}[htbp]
\centering
\begin{tabular}{|l|l|l|}
\hline
\textbf{Algorithm} & \textbf{Iter} & \textbf{std}  \\ \hline
$d= \,~1$ & ~~1,476 & \, 67.18 \\ \hline
$d= \,~5$ & ~~6,707 & 315.02 \\ \hline
$d=10$ & 10,015 & 540.10  \\ \hline
$d=20$ & 14,469 & 749.89  \\ \hline
\end{tabular}
\caption{First iteration when \cris{the value of $f_{t}(x_{t})-f_{t}(x^{*}_{*})$} is less than $0.1$ (iter), along with the standard deviation of the final iterate (std), as shown in Figure \ref{fig:noteo}.}

\label{tab:Alg1}
\end{table}

\begin{figure}[htbp]
\centering
\includegraphics[scale=0.5]{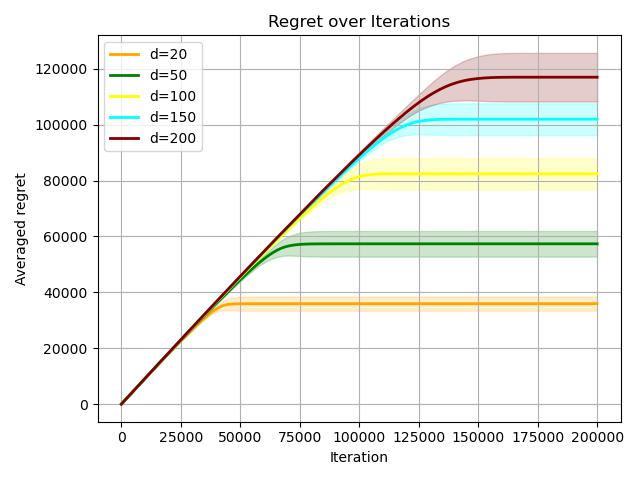}
\caption{Average regret of Algorithm \ref{DOGD-SC} on the Lipschitz continuous function defined in \eqref{ex:function1} in a high-delay setting with $d \in \{20, 50, 100, 150, 200\}.$}
\label{fig:noteo1}
\end{figure}
\begin{table}
\centering
\begin{tabular}{|l|l|l|}
\hline
\textbf{Algorithm} & \textbf{Iter} & \textbf{std}  \\ \hline
$d= \,~20$ & ~~45,694& 2,375 \\ \hline
$d= \,~50$ & ~~71,466 & 3,746 \\ \hline
$d=100$ & 103,445 & 6,024  \\ \hline
$d=150$ & 127,445 & 5,842  \\ \hline
$d=200$ & 142,963 & 5,784  \\ \hline
\end{tabular}

\caption{First iteration when \cris{the value of $f_{t}(x_{t})-f_{t}(x^{*}_{*})$} is less than $0.1$ (iter), along with the standard deviation of the final iterate (std), as shown in Figure \ref{fig:noteo1}.}
\label{tab:Alg1.5}
\end{table}

\noindent Figure \ref{fig:noteo} illustrates the performance of Delayed Online Gradient Descent on highly non-convex functions. As observed, the regret curves grow sub-linearly with the number of iterations and flatten after a certain point. Additionally, the regret increases with the delay, meaning that larger delays result in higher regret and require more iterations to flatten the curve, as shown in Table \ref{tab:Alg1}. These results support our theoretical prediction that larger delays lead to greater regret. A similar trend is observed for the standard deviation, which also increases as the delay grows. \medskip

\noindent We also conduct experiments in a high-delay setting, where $d \in \{20, 50, 100, 150, 200\}$ and $T = 200,000$, while keeping the other parameters the same as in the previous experiment. As seen in Figure \ref{fig:noteo1} and Table \ref{tab:Alg1.5}, the regret curves grow sub-linearly with the number of iterations and then stabilize. The delay causes the regret to increase, and the number of iterations required for stabilization also grows, supporting our earlier findings. However, in this case, the standard deviation does not necessarily increase with the delay. Note that although the same class of functions and parameters is used, this experiment is not directly comparable to the previous one since the step-size decreases as $T$ increases. For instance, when $d = 20$, the regret curve stabilizes at iteration $45{,}203$, compared to $14{,}469$ in the previous experiment, due to the smaller step-size (reduced by a factor of $\sqrt{10}$), highlighting the importance of the step-size choice. \medskip

\noindent For the second experiment, we consider the function presented in \eqref{ex:f2}, which is weakly smooth. We set $R=1$ and $p=100$. The sequence of optimal solutions is initialized as $x_1^{*} \sim \mathcal{N}(0, \Id)$ and evolves, for every $t\in[T-1]$, according to the update rule:
\begin{align*}
x_{t+1}^{*} = x_{t}^{*} + \frac{0.1}{t^{\frac{1}{2}}} v_{t},
\end{align*}
where $v_1, \dots, v_{T-1} \sim \mathcal{N}(0, \Id)$ are independent and identically distributed standard Gaussian vectors. This implies that $f_{t}(x_{t}^{*}) = 0$ for every $t \in [T]$, and that $V_T = \mathcal{O}(T^{\frac{1}{2}})$. The initial iterate is sampled as $x_{0} \sim \mathcal{N}(0, \Id)$ and then normalized. For every $t\in [T]$, we collect $m = 1000$ independent and identically distributed samples $\{(a_{t}^{i}, b_{t}^{i})\}_{i=1}^{m}$, where the inputs are $a_{t}^{i} \sim \mathcal{N}(0, \Id)$, and the outputs are defined as $b_{t}^{i} = \sigma(\scal{a_{t}^{i}}{x^{*}_{t}})$. The step-size is set as
\begin{align}\label{step-size1}
\eta_{t} = 0.99 \frac{2\kappa}{(d + 2d^{\frac{1}{2}}(d-1))\Gamma}.
\end{align}

\noindent Figure \ref{fig:noteo2} illustrates the performance of Delayed Online Gradient Descent on GLM. As observed in the previous experiment, both the regret and the standard deviation increase with the delay, and all regret curves grow sub-linearly. In this example, when there is no delay, the cumulative regret is significantly lower than in the delayed cases and stabilizes much faster. This highlights the impact of delay on cumulative regret, as larger delays lead to both higher regret and slower stabilization.\\

 \begin{figure}[htbp]
\centering
\includegraphics[scale=0.5]{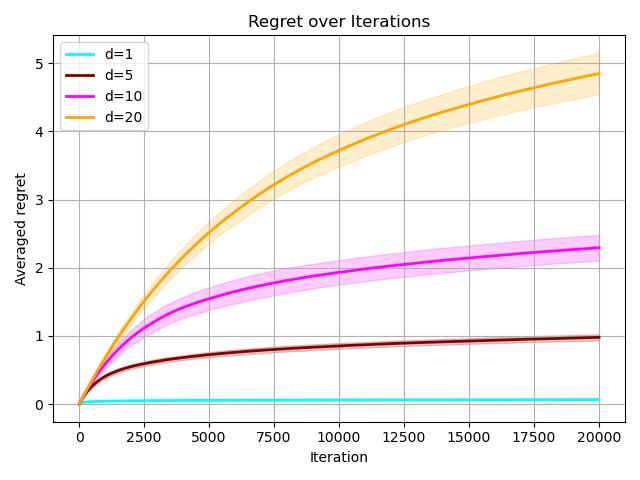}
\caption{Average regret of Algorithm \ref{DOGD-SC} on the weakly smooth function defined in \eqref{ex:f2}.}
\label{fig:noteo2}
\end{figure}
\begin{table}
\centering
\begin{tabular}{|l|l|l|}
\hline
\textbf{Algorithm} & \textbf{Iter} & \textbf{std}  \\ \hline
$d=~\,1$ & ~~~~~~78 & 0.029 \\ \hline
$d=~\,5$ & ~~1,862 & 0.053 \\ \hline
$d=10$ & ~~5,328 & 0.127 \\ \hline
$d=20$ & 15,728 & 0.366  \\ \hline
\end{tabular}
\caption{First iteration when \cris{the value of $f_{t}(x_{t})-f_{t}(x^{*}_{*})$} is less than $10^{-4}$ (iter), along with the standard deviation of the final iterate (std), as shown in Figure \ref{fig:noteo2}.}
\label{tab:Alg2}
\end{table}

\noindent To explore the impact of path variation $V_T$, we keep $d$, $\Gamma$, $\kappa$, $\theta$, and $R$ constant while varying $V_T$. Consequently, the upper bound on the regret depends only on $V_T$. To achieve this, we modify the recursion for $x_t^{*}$ by introducing a parameter $\text{a}$ as follows:  
$$x_{t+1}^{*} = x_{t}^{*} + \frac{0.1}{t^{\text{a}}} v_{t},$$  
where $\text{a} \in \{0.0625, 0.125, 0.25, 0.5, 1\}$, and $\{v_t\}_{t \in [T]}$ is sampled as in the previous experiment. This implies that $V_T = \mathcal{O}(T^{1-\text{a}})$. The initial iterate $x_0$ is sampled as in the previous experiment, and for each $t \in [T]$, $\eta_t$, $b_t$, and $A_t$ are sampled as in the previous experiment. Even though we are comparing different functions, by normalizing $A_t$ and $b_t$, we ensure that all parameters except $V_T$ remain the same, allowing us to study the effect of $V_T$ on the upper bound of the regret.

\begin{figure}[htbp]
\centering
\includegraphics[scale=0.5]{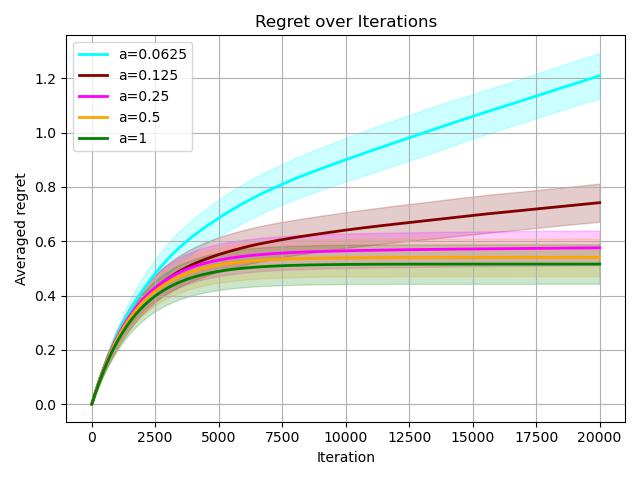}
\caption{Average regret of Algorithm \ref{DOGD-SC} on the weakly smooth function defined in \eqref{ex:f2} for differents $V_{T}$.}
\label{fig:noteo3}
\end{figure}
\begin{table}
\centering
\begin{tabular}{|l|l|l|}
\hline
\textbf{Algorithm} & \textbf{Iter} & \textbf{std}  \\ \hline
$\text{a} = 0.0625$ & 2,939& 0.084 \\ \hline
$\text{a} = 0.125$ & 2,135 & 0.070 \\ \hline
$\text{a} = 0.25$ & 2,021 & 0.064  \\ \hline
$\text{a} = 0.5$ & 1,945 & 0.068  \\ \hline
$\text{a} = 1$ & 1,867 & 0.073  \\ \hline
\end{tabular}
\caption{First iteration when \cris{the value of $f_{t}(x_{t})-f_{t}(x^{*}_{*})$} is less than $10^{-4}$ (iter), along with the standard deviation of the final iterate (std), as shown in Figure \ref{fig:noteo3}.}
\label{tab:Alg3}
\end{table}

\noindent As seen in Figure \ref{fig:noteo3}, the regret decreases as the value of $\text{a}$ increases. All regret curves, except for $\text{a} = 0.0625$, exhibit sub-linear growth, indicating that they may require more iterations to flatten rather than suggesting linear growth. In Table \ref{tab:Alg3}, we observe that, as $\text{a}$ increases, \cris{the gap $f_{t}(x_{t})-f_{t}(x^{*}_{*})$} decreases more quickly. However, the standard deviation remains consistent across different values of $\text{a}$. These results reinforce our findings on the impact of cumulative path variation on regret.\\

\noindent For the third experiment, we consider the function defined in \eqref{ex:f3}. We set $B_{t} \equiv 0$, $R = 10$, and $p = 50$. The parameters of the function $f_{t}$ are initialized as follows: $A_{0}$ is a positive definite matrix with norm $1$, while $a_{0}$ and $b_{0}$ are independently sampled from a standard normal distribution and then normalized to have norm $0.1$.

\noindent For each $t \in [T]$, the parameters are updated according to the following recursions:
\begin{align*}
    A_{t} &= \frac{A_{t-1} + 0.01\, t^{-\tfrac{1}{2}} v_{t}^{1}}{\|A_{t-1} + 0.01\, t^{-\tfrac{1}{2}} v_{t}^{1} \|}, \hspace{3mm}
    a_{t} = \frac{a_{t-1} + 0.01\, t^{-\tfrac{1}{2}} v_{t}^{2}}{10 \|a_{t-1} + 0.01\, t^{-\tfrac{1}{2}} v_{t}^{2}\|}, \hspace{3mm} \text{and} \hspace{3mm}
    b_{t} = \frac{b_{t-1} + 0.01\, t^{-\tfrac{1}{2}} v_{t}^{3}}{10 \|b_{t-1} + 0.01\, t^{-\tfrac{1}{2}} v_{t}^{3}\|},
\end{align*}
where $v_{t}^{1}$ is a positive definite matrix, and $v_{t}^{2}$ and $v_{t}^{3}$ are independent and identically distributed standard Gaussian vectors. We also set $\beta_{t} \equiv \|b_{t}\|R + 100$ and $\alpha_{t} = 10$, which ensures that, for all $x \in \XX$, $x \in K$ with $m = 100$ and $M = \|b_{t}\|R + \beta$. We set the step-size $\eta$ as in \eqref{step-size1}. Since the sequence of optimal solutions is not available a priori (and consequently neither is $V_{T}$), we implement a gradient descent subroutine at each iteration $t \in [T]$, which, as guaranteed by Corollary \ref{cor1}, converges to the unique solution. The stopping criterion for this subroutine is based on the relative error, with a tolerance of $10^{-6}$. To improve execution time, we adopt a warm-start strategy, initializing the subroutine at iteration $t$ with the solution from iteration $t-1$.

\begin{figure}[htbp]
\centering
\includegraphics[scale=0.5]{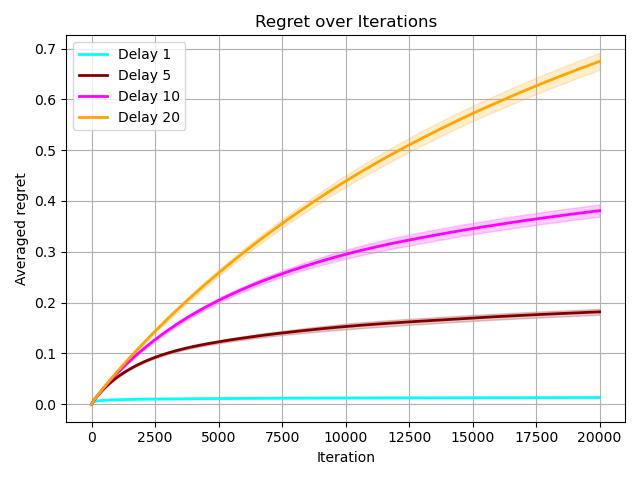}
\caption{Average regret of Algorithm \ref{DOGD-SC} on the weakly smooth function defined in \eqref{ex:f3} for differents $V_{T}$.}
\label{fig:noteo4}
\end{figure}
\begin{table}
\centering
\begin{tabular}{|l|l|l|}
\hline
\textbf{Algorithm} & \textbf{Iter} & \textbf{std}  \\ \hline
$d=~~1$ & ~~~~162& 3.1$\cdot 10^{-4}$ \\ \hline
$d=~~5$ & ~~4126 & 5.9$\cdot 10^{-3}$ \\ \hline
$d=10$ & 12069 & 1.2$\cdot 10^{-2}$ \\ \hline
$d=20$ & 19985 & 1.7$\cdot 10^{-2}$  \\ \hline
\end{tabular}
\caption{First iteration when \cris{the value of $f_{t}(x_{t})-f_{t}(x^{*}_{*})$} is less than $10^{-5}$ (iter), along with the standard deviation of the final iterate (std), as shown in Figure \ref{fig:noteo4}.}
\label{tab:Alg4}
\end{table}

\noindent Figure \ref{fig:noteo4} illustrates the performance of Delayed Online Gradient Descent on quadratic fractional functions. As observed in the previous experiments, the regret curves grow sub-linearly with the number of iterations and then stabilize, flattening after a certain number of iterations. As in the previous experiments the regret increases with the delay. Table \ref{tab:Alg4} shows that larger delays require more iterations for the curve to stabilize. A similar behavior is observed for the standard deviation. \medskip

\noindent To investigate the impact of error when the full gradient is unavailable, and consequently how $\Lambda_{T}$ and $\Delta_{T}$ affect the regret, we set the parameters as in the previous experiment, but fix $d=5$ since the curve stabilizes faster than for $d=10$ in the previous experiment, which allows a wider range for selecting the parameter $\text{a}$. In this case, we define $r_{k}$ as in \eqref{zeroth order} with $h_{t}=\tfrac{1}{t^{\text{a}}}$, where $\text{a} \in \{0.4, 0.6, 0.8, 1\}$. As a result, the upper bound on the regret depends only on $\Lambda_{T}$ and $\Delta_{T}$, allowing us to analyze their effects on the regret’s upper bound. We also compare the results with the case when the full gradient is available ($h_{t}\equiv 0$). This leads to $\Delta_{T} \geq \mathcal{O}(T^{1-\text{a}})$ for $\text{a} \in \{0.4, 0.6, 0.8\}$, and $\Delta_{T} \geq \mathcal{O}(\ln(T))$ when $\text{a}=1$.

\begin{figure}[htbp]
\centering
\includegraphics[scale=0.35]{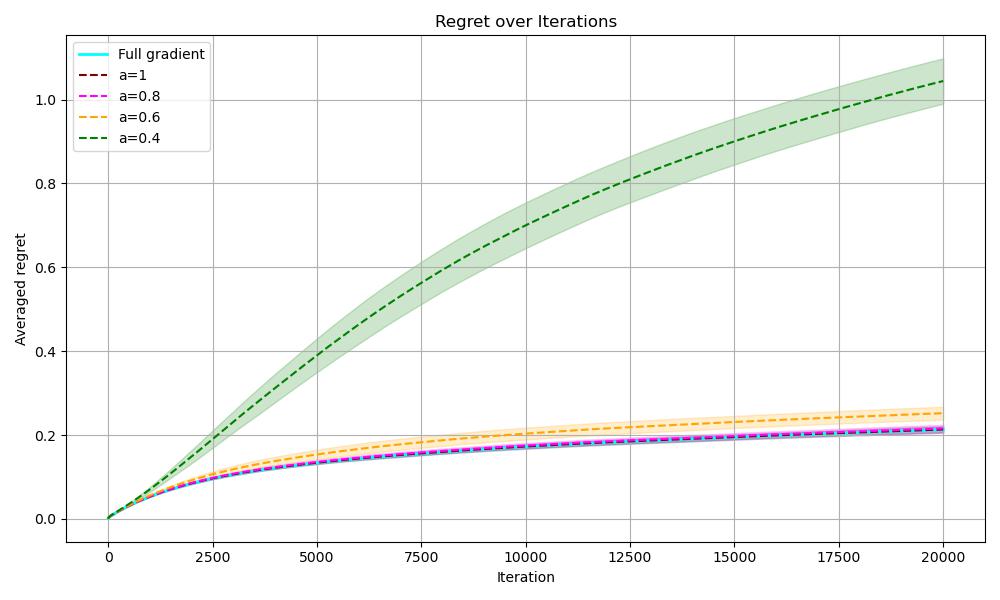}
\caption{Average regret of Algorithm \ref{DOGD-SC} with $d=5$ on the weakly smooth function defined in \eqref{ex:f3} for differents $\Delta_{T}$ and $\Lambda_{T}$.}
\label{fig:noteo6}
\end{figure}
\begin{table}
\centering
\begin{tabular}{|l|l|l|l|}
\hline
\textbf{Algorithm} & \textbf{Iter} & \textbf{std} & \textbf{Time} [s]\\ \hline
$\text{Full gradient}$ & 5394& $7.6\cdot 10^{-3}$ & ~~1.0  \\ \hline
$\text{a} = 1$ & 5,394 & $7.8\cdot 10^{-3}$ & 19.9\\ \hline
$\text{a} = 0.8$ &5,495 & $8.5\cdot 10^{-3}$&  19.7\\ \hline
$\text{a} = 0.6$ & 6,853 & $1.5\cdot 10^{-2}$& 19.5\\ \hline
$\text{a} = 0.4$ & - & $5.4\cdot 10^{-2}$& 19.6\\ \hline

\end{tabular}
\caption{First iteration where \cris{the value of $f_{t}(x_{t})-f_{t}(x^{*}_{*})$} falls below $10^{-5}$ (iter) and the standard deviation of the final iterate (std), as illustrated in Figure \ref{fig:noteo6}. The average execution time for each experiment is also included.}
\label{tab:Alg6}
\end{table}

\noindent As illustrated in Figure \ref{fig:noteo6}, the regret decreases as the value of $\text{a}$ increases. For $\text{a} = 1$ and $\text{a} = 0.8$, the regret curve closely follows the full gradient curve, indicating a good approximation. A similar behavior occurs when $\text{a} = 0.6$, though the gradient approximation is less accurate compared to the previous cases. When $\text{a} = 0.4$, the regret increases, and the gap $f_{t}(x_{t}) - f_{t}(x^{*}_{*})$ does not drop below $10^{-5}$ within 20,000 iterations. Table \ref{tab:Alg6} shows that as the value of $\text{a}$ increases, fewer iterations are required to reach the given tolerance level. A similar trend is observed for the standard deviation. Furthermore, the standard deviation for the full gradient and for $\text{a} = 1$ is comparable, though the gradient computation in the bandit setting incurs a significantly higher execution time. This raises the question of exploring cheaper gradient approximations, such as the two-point estimator proposed in \cite{agarwal2010optimal,wan2022online}.

\section{Conclusions and Future Work}\label{sec:cfw}

In this work, we introduce a projected gradient descent algorithm to address online quasar-convex optimization in a delayed feedback setting for both weakly smooth and Lipschitz gradient functions. We establish a sub-linear dynamic regret rate for the proposed algorithm, which depends on the cumulative path variation, the delay, the radius of the set, and the time horizon. Furthermore, we extend our algorithm to the (zeroth-order) bandit setting and establish the convergence of the delayed zeroth-order projected gradient descent method for strongly quasar-convex functions. We also provide new examples of quasar-convex functions that improve the applicability of online optimization to quasar-convex functions, such as strongly quasiconvex functions. Experimental results validated our theoretical findings, highlighting the influence of path variation and delay on algorithmic behavior. \medskip

\noindent Future research will extend this work to more general settings such as mirror descent, inertial schemes, second-order methods, and non-smooth functions. Another research direction is to investigate stochastic error in order to incorporate more complex zeroth-order schemes, particularly the two-point estimator in \cite{flaxman2004online,wan2022online}. We also plan to study how to weaken conditions to consider unbounded feasible sets and thus consider linear constraints.

\section{Declarations}







\subsection{Availability of supporting data}

No data sets were generated during the current study. The used {\sc python} codes are available from all authors on reasonable request.

\subsection{Author Contributions}

 All authors contributed equally to the study conception, design and implementation and wrote and corrected the manuscript.

\subsection{Competing Interests}

There are no conflicts of interest or competing interests related to this
manuscript.

\subsection{Funding}

 This research was partially supported by UTA research project Fortalecimientos Grupos de Investigaci\'on C\'odigo 8802-25 and by ANID---Chile through Fondecyt Regular 1241040 (Lara).

\subsection{Acknownledgements:}

{\color{black} The authors wishes to thank the reviewer for their comments and suggestions that helped to improve the paper.}

\printbibliography
 \end{document}